\documentclass[a4paper,11pt]{amsart}

	\usepackage[utf8]{inputenc}
	\usepackage[T1]{fontenc}
	\usepackage[english]{babel}
	\usepackage[protrusion=true,expansion=true,kerning, spacing]{microtype}
	\microtypecontext{spacing=english}
	\usepackage{CJKutf8} 
	\usepackage{cmap} 

	\usepackage{xspace}
\usepackage{booktabs} 
	
	\usepackage[textwidth=16cm,hcentering,heightrounded,foot=1cm]{geometry}
	\usepackage{caption}
 \usepackage{titlesec}
	\titleformat{name=\section}[hang]{\scshape \large \centering}{\thesection.~}{0em}{}[]

	\titleformat{\subsection}[hang]{\bfseries}{\thesubsection.~}{0em}{}[]

	\titleformat{\subsubsection}[runin]{\bfseries}{\thesubsubsection.}{.2em}{}[]

 \usepackage{titletoc}
\titlecontents{section}[2pc]{\addvspace{1pc}}
{\contentslabel[\textsc{\thecontentslabel}]{1pc}}
{}{\dotfill\contentspage}
[\addvspace{2pt}]

\titlecontents{subsection}[4pc]{}
{\contentslabel[\subsectionname\thecontentslabel]{2pc}}
{}{\hfill\contentspage}
[]

\usepackage{csquotes}
\usepackage[
bibencoding=utf8,%
giveninits=true,
bibstyle=ieee-alphabetic,
citestyle=alphabetic,%
sorting=nyt,%
language=french,%
isbn=false,%
doi=false,%
eprint=true,
backref=true, 
hyperref=true,
url=true,
minbibnames=4,
maxbibnames=4,
dashed=true,
backend=biber,
]{biblatex}
	\usepackage{amsthm,amssymb,amsmath}
	\usepackage{amssymb,amsfonts}
	\usepackage[fixamsmath]{mathtools} 
	\usepackage{accents} 
	\usepackage{tikz-cd}
		\numberwithin{equation}{subsection}

	\usepackage{graphics,graphicx}
	\usepackage{color}
	\usepackage{tikz}
	\usepackage{caption} 
	\usepackage{subcaption} 
	\usepackage{wrapfig}

\theoremstyle{definition}
	\newtheorem{defin}{Definition}[section]

\theoremstyle{definition}
	\newtheorem{rem}[defin]{Remark}
	
\theoremstyle{plain}
	\newtheorem{theo}[defin]{Theorem}
	\newtheorem{prop}[defin]{Proposition}
	\newtheorem{cor}[defin]{Corollary}
	\newtheorem{lem}[defin]{Lemma}

\newcommand{\N}{\mathbf{N}}

\newcommand{\Q}{\mathbf{Q}}
\newcommand{\Z}{\mathbf{Z}}

\newcommand{\F}{\mathbf{F}}

\newcommand{\PP}{\mathbf{P}} 
\newcommand{\G}{\mathbf{G}}

\newcommand{\OO} {{\mathcal O}} 
\newcommand{\mA} { \mathcal{A}}

\newcommand{\mB} {\mathcal{B}}

\newcommand{\mT} {\mathcal{T}}

\newcommand{\mC} {\mathcal{C}}

\newcommand{\mM} {\mathcal{M}}

\newcommand{\mfS} {\mathfrak{S}}

\newcommand{\Ker}{\operatorname{Ker}}
\newcommand{\Spec} { \operatorname{Spec} }

\newcommand{\End} {\operatorname{End} }

\newcommand{\Card} {\operatorname{Card} }
\newcommand{\Aut} {\operatorname{Aut} }
\newcommand{\Gal} {\operatorname{Gal} }

\newcommand{\GL}{\operatorname{GL}}

\newcommand{\Tl}{\operatorname{T}_{\ell}}

\newcommand{\Sp}{\operatorname{Sp}}

\usepackage[pdftex,
hidelinks,
hyperindex=true, bookmarks=true, bookmarksopen=true, bookmarksnumbered=true,
pdfauthor={Séverin Philip},
pdftitle={Philip - Stability degree of curves},
pdfsubject={},
pdfstartview=FitH
]{hyperref}

\begin{document}
\author{Séverin Philip}
\address{Department of Mathematics, Stockholms universitet
SE-106 91 Stockholm, Sweden}
\thanks{The author would like to thank J.~Bergström and A.~Tamagawa for the many discussions around the topics of stable reduction of curves and Krasner's lemma. The author is supported by the Verg Foundation.}
\keywords{curves, semi-abelian varieties, stable reduction}
	\email{severin.philip@math.su.se}
\title[Finite monodromy groups in degenerating families]{Finite monodromy groups in degenerating families and the stability degree of curves \\[0.5em]
\textnormal{ Version of \today}}

\maketitle

\begin{abstract}
    We compute the stability degree of curves of genus $g$ and show that it is equal to the classical Minkowski bound $M(2g)$. We also compute the stability degree of semi-abelian varieties with prescribed toric and abelian ranks. The proof uses a specialization result for finite monodromy groups in degenerating families, obtained from a variant of Krasner's lemma, together with an explicit construction at the prime $2$.
\end{abstract}

\section{Introduction}

\stepcounter{subsection}

\subsubsection{} The notion of stable reduction of curves has been introduced by Deligne and Mumford in their seminal paper \cite{DM72} in order to provide a compactification of the moduli space of curves. For a smooth curve $C$ of genus $g\geq 1$ over a number field there is a finite extension such that the base change of $C$ has stable reduction. Our first result is to compute the value of the smallest integer, denoted $d_g^{\mathrm{Cur}}$, such that for any curve of genus $g$, there is an extension of degree at most $d_g^{\mathrm{Cur}}$ for which the base change of the curve has stable reduction.

\begin{theo}[Theorem~\ref{theo:dgc}]
    Let $g$ be a positive integer. We have 
    \[
     d_g^{\mathrm{Cur}} = M(2g)
    \]
    where $M(n)=\prod_{p \text{ prime}} p^{r(n,p)}$ with $r(n,p)=\sum_{i\geq 0 } \lfloor n/p^i(p-1)\rfloor$.
\end{theo}

Note that the $g=1$ case is already known from the main result of \cite{Phi22a}.

\medskip

Similar questions were studied by Chrétien, Lehr and Matignon in a series of papers \cite{CM13, LM05, LM06} and an unpublished manuscript. Their constructions are explicit and provide examples of curves with $p$-maximal stability degree in low genus. 

\subsubsection{} Deligne and Mumford show that stable reduction of a curve is equivalent to semi-stable reduction for its Jacobian so that the upper bound follows directly from earlier work -- see \cite{Phi22a}. To obtain the lower bound we use singular stable curves, defined as trees whose leaves are Matsumoto-Seyama curves of \cite{CP25}. By Galois twisting, we obtain curves whose Jacobians have a maximal semi-stability degree at odd primes $p$. For the even prime, we use a new ad hoc construction, first of an abelian variety with maximal finite monodromy at $2$ given in Section~\ref{sub:adhocabvar} which we then show in Section~\ref{sub:curvesconstr} that it is the Jacobian of a stable curve. Since those curves are singular their Jacobians are semi-abelian varieties with positive toric rank in most cases. The first part of the paper is thus devoted to studying the semi-stable reduction of semi-abelian varieties and their finite monodromy groups, the groups that represent the local obstruction to semi-stable reduction. The proofs extend those of \cite{Phi22a} and \cite{Ph22b} from abelian varieties to semi-abelian varieties with special care of the toric case. A Grothendieck type characterization of semi-stable reduction for semi-abelian varieties is given in \cite{HN10}, we provide a similar proof for the sake of completeness and some details that are relevant to our work. We give a computation of the semi-stability degree $d_g(t,a)$ of semi-abelian varieties with toric rank $t$ and abelian rank $a$ -- see Definition~\ref{def:degdgta} -- and a Silverberg-Zarhin type description of their finite monodromy groups in Theorem~\ref{theo:(p,t,t_0,a_0)-inertial}.

\smallskip

Let $\beta\colon \N\to \Q$ be defined by $\beta(2)=3/2$, $\beta(4)=3$, $\beta(6)=9/4$, $\beta(7)=9/2$, $\beta(8)=135/2$, $\beta(9)=15/2$, $\beta(10)=9/4$ and $\beta(n)=1$ for all other values of $n\in \N$.

\begin{theo}[Theorem~\ref{theo:dgta}]
    For all integers $g\geq 1$, $t\geq 0$, and $a\geq 0$ such that $g=t+a$ we have 
    \[d_g(t,a)= \beta(t)2^tt!\cdot M(2a).
    \]    
\end{theo}

\subsubsection{} The main technical difficulty is to relate the results on semi-abelian varieties to the stable reduction of smooth curves by the use of singular curves. To do this, we provide a new form of Krasner's lemma, which deals with quasi-finite morphisms over local fields. Precisely, it generalizes earlier versions of Krasner's lemma for finite étale maps by providing that, in the smooth locus of the base, points outside the étale locus provide lower bounds for the Galois groups of the étale fibers that are close enough.
\begin{theo}[See Theorem~\ref{theo:krasner}]
    Let $f\colon X\to Y$ be a quasi-finite, flat and separated morphism of varieties over a $p$-adic field $K$. Assume $Y$ is smooth and let $S\subset Y$ be the non-empty open subset such that $f$ is finite étale. Then there is a finite open and closed covering $U_1,\dots, U_n$ of $S(K)$ and finite groups $G_1,\dots, G_n$ such that, for all $1\leq i\leq n$, $s\in S(K)$ is in $U_i$ if and only if the Galois group of the fiber of $f$ at $s$ is $G_i$. Furthermore, for $x\in (Y\setminus S) (K)$, there is an $i\in \{1,\dots, n\}$ such that $x\in \overline{U_i}$ and for any such $i\in \{1,\dots, n\}$ we have a surjection $G_i\to G_x$, with $G_x$ the Galois group of the fiber at $x$.
\end{theo}

Applied to the moduli spaces of stable curves with enough marked points, we get that the singular curves provide lower bounds to the stability degree of the smooth curves that are close enough. 

\subsubsection{} The first part of the paper is devoted to the study of the finite monodromy groups and the semi-stable reduction of semi-abelian varieties. The results build on those of \cite{Phi22a}, \cite{Ph4} and \cite{Ph7} mostly. In order to compute the degree of semi-stability $d_g(t,a)$ special attention to the case of tori is given in Section~\ref{sub:semi-stabdeg}. In order to deal with those, we make use of the maximal subgroups of $\GL_t(\Z)$ given by Feit, in \cite{Fei96}, with a twisting result for tori. 

\medskip

The second part deals with our new version of Krasner's lemma over complete fields with discrete valuations of rank $1$. We first go over some standard facts in the analytic topology about finite étale maps, namely that they are open and closed, and about the rational points in the smooth locus. We are then able to prove our Krasner's lemma in its general form before making a special case for the inertia subgroups. The proof is straightforward and is obtained by carefully studying the properties of the sets of rational points with fibers having a given Galois group or being included in given field extension of the base field. To apply this to semi-abelian schemes that degenerate, the key is Lemma~\ref{lem:bigltorfinimonod} which shows that one can identify the finite monodromy groups of the fibers of the semi-abelian scheme $G\to S$ from the fibers of the $\ell$-torsion map $G[\ell]\to S$. Krasner's lemma applied with this map thus provides open subsets of $S(K)$ with fixed finite monodromy, bounded by their boundaries. 

\medskip

The last part of the paper is the application to the stability degree of curves. We first construct the singular curves with maximal finite monodromy groups for odd $p$ using the Matsumoto-Seyama curves of \cite{CP25} and the twisting construction of \cite{Ph22b}. We start by a digression, providing a construction of curves in characteristic $2$ with automorphism groups having their order with maximal $2$-adic valuation. The construction is detailed to showcase the slightly more difficult situation of characteristic $0$ and odd primes. For the main result, Krasner's lemma in the form of Theorem~\ref{theo:recouvkrasnersab} is applied to the Jacobian of the biggest open subscheme of the compactified moduli spaces of curves with many marked points. We thus get an incomplete Skolem data in the sense of Moret-Bailly -- see \cite{MB89} -- on the locus of smooth curves in the moduli space which thus has an integral solution. The resulting curve provides the lower bound to $d_g^{\mathrm{Cur}}$.  

\section{The case of semi-abelian varieties}

\subsection{A characterization of semi-stable reduction}

\subsubsection{} We will prove a characterization of semi-stable reduction of semi-abelian varieties by their $\ell$-adic Tate modules. This was already done by Halle and Nicaise in \cite{HN10} Section~4 and our proof is essentially the same but gives slightly more details. We record it here as it is useful context for the next parts. 

\medskip

Let $A$ be a semi-abelian variety over the maximal unramified extension $K$ of a $p$-adic field. The semi-abelian variety $A$ fits into the following exact sequence
\[
\begin{tikzcd}
    0 \arrow[r] & T \arrow[r] & A\arrow[r] & B \arrow[r] & 0
\end{tikzcd}
\]
where $T$ is a torus and $B$ an abelian variety over $K$. Let us denote $g=\dim A$, $t=\dim T$ and $a=\dim B$.

\medskip

Let $\ell$ be a prime  number distinct from the residual characteristic $p$ of $K$. Let us denote by $\Tl T$, $\Tl A$ and $\Tl B$ the $\ell$-adic Tate modules of $T$, $A$, and $B$ respectively. These are free $\Z_{\ell}$-modules equipped with an action of $I_K$, the absolute Galois group of $K$. Let us recall that from Grothendieck's monodromy theorem for $\ell$-adic representations there is a smallest extension $L/K$ such that $I_L$ acts unipotently on $\Tl A$. This action of $I_L$ factors through its quotient isomorphic to $\Z_{\ell}$ and is thus given by the action of a single element. By the exact sequence defining $A$, we have an exact sequence of $I_K$-modules

\[
\begin{tikzcd}
    0 \arrow[r] & \Tl T \arrow[r] & \Tl A\arrow[r] & \Tl B \arrow[r] & 0.
\end{tikzcd}
\]

It follows that the action of $I_L$ on $\Tl T$ is trivial and it is unipotent on $\Tl B$. The semi-abelian variety $A$ admits a Néron model $\mA$ over the spectrum $S$ of the ring of integers of $K$. The same holds for $T$ and $B$ and we denote their Néron models by $\mT$ and $\mB$ respectively. We also denote with $\circ$ the subscheme with special fiber the connected component of the neutral point. Finally, we denote by $A_k$, $T_k$ and $B_k$ the reductions of $A$, $T$ and $B$, that is the connected components of the special fibers of their Néron models.

\medskip

The following lemma is extracted from the proof of Proposition~10.1.7 of \cite{BLR90}. 
 
\begin{lem} \label{lem:exactconnectedneron}
    Let $A$ be a semi-abelian variety with maximal torus $T$ and abelian quotient $B$. Suppose that $T$ is split. Then there is an exact sequence of $S$-schemes
    \[
    \begin{tikzcd}
        0\arrow[r] & \mT^{\circ} \arrow[r] & \mA^{\circ} \arrow[r] & \mB^{\circ} \arrow[r] & 0
    \end{tikzcd}
    \]
\end{lem}

\medskip

We also need the following lemma to translate this result to $\ell$-adic Tate modules.

\begin{lem} \label{lem:neronmodels}
    Let $L/K$ be a finite extension such that $I_L$ acts unipotently on $\Tl A$. Then we have an exact sequence
    \[
    \begin{tikzcd}
    0 \arrow[r] & (\Tl T)^{I_L} \arrow[r] & (\Tl A)^{I_L}\arrow[r] & (\Tl B)^{I_L} \arrow[r] & 0.
    \end{tikzcd}
    \]
    
\end{lem}
\begin{proof}
 Since the action of $I_L$ is unipotent on $\operatorname{T}_{\ell} A$, it is trivial on the stable sublattice corresponding to $T$ -- this comes down to the fact that $T$ is a twist so that the Galois action on its Tate module is trivial after a finite extension. We thus have that $T_L$ is split and we can apply the previous lemma to the situation over $L$. Let us denote by $\mT'$, $\mA'$ and $\mB'$ the Néron models over $\OO_L$ of respectively $T_L$, $A_L$ and $B_L$. Then, the sequence
 \[
    \begin{tikzcd}
        0\arrow[r] & \mT'^{\circ} \arrow[r] & \mA'^{\circ} \arrow[r] & \mB'^{\circ} \arrow[r] & 0
    \end{tikzcd} 
 \]
of algebraic groups over $\OO_L$ is exact. 

\medskip

We have that $A[\ell^n]^{I_L}$ is identified with the subgroup scheme of $\mA'[\ell^n]$ which is finite étale over $\Spec \OO_L$ -- see \cite{sga} Exposé IX, section 2.2.2. An element of $( \operatorname{T}_{\ell} A)^{I_L}$ is thus a compatible sequence of elements in such torsion group schemes. 

\medskip

From the exactness of the sequence it suffices to show that for an element $(x_i)\in( \operatorname{T}_{\ell} A)^{I_L}$ we have $x_i \in \mA'^{\circ}$ for all $i$ to conclude. This can be checked on the special fiber as for $i\in \N$, $x_i$ corresponds to a point $y_i$ on the special fiber. Note that the point $y_i$ is infinitely divisible by $\ell$. Let $n\in \N$ such that $\ell^n$ is greater than the order of the group of connected components of $\mA'_k$. Then $y_i=\ell^n \cdot y_{i+n}$ so that $y_i\in \mA_k^{'\circ}$ and $x_i\in \mA^{'\circ}$ by definition. 
\end{proof}

We get the following corollary.

\begin{cor}
    For $\ell$ big enough, the sequence
    \[
    \begin{tikzcd}
    0 \arrow[r] & (\Tl T)^{I_K} \arrow[r] & (\Tl A)^{I_K}\arrow[r] & (\Tl B)^{I_K} \arrow[r] & 0
    \end{tikzcd}
    \]
    is exact.
\end{cor}
\begin{proof}
    First recall that we have the equality $M^{I_K}=(M^{I_L})^{\Gal(L/K)}$ for all finite Galois extensions $L$ of $K$ and all $I_K$-module $M$. By the previous lemma, the sequence
    \[
    \begin{tikzcd}
    0 \arrow[r] & (\Tl T)^{I_L} \arrow[r] & (\Tl A)^{I_L}\arrow[r] & (\Tl B)^{I_L} \arrow[r] & 0
    \end{tikzcd}
    \]
    is exact, for any extension $L/K$ such that $I_L$ acts unipotently on $\Tl A$. By Grothendieck's monodromy theorem such an $L$ exists and we choose one. We thus have to take the $\Gal(L/K)$-invariants of the previous exact sequence. But if $\ell$ is big enough, the order of the finite group $\Gal (L/K)$ is invertible  in $\Z_{\ell}$. It follows that, as a free $\Z_{\ell}$-module, $(\Tl T)^{I_L}$ is cohomologically trivial for $\Gal(L/K)$. Thus we have the exactness of the sequence of invariants we sought. 
\end{proof}

Finally, let us show that the unipotent rank $\lambda_A$ of the reduction $A_k$ of $A$ is the sum of the unipotent ranks $\lambda_T$ and $\lambda_B$ of the reductions of $T$ and $B$. We also denote by $a_0$ and $t_0$ the abelian and toric ranks of the reduction $A_k$ as well as $t'$ and $a'$ the toric and abelian ranks of $B_k$. Finally denote by  $t''$ the toric rank of $T_k$. 

\begin{prop}  \label{prop:reduccomparrangs}
    We have $\lambda_A= \lambda_T+\lambda_B$, $a_0=a'$ and $t_0=t'+t''$.
\end{prop}
\begin{proof}
     By definition we have the equalities $t'+a'+\lambda_B=\dim B_k=\dim B=a$, $t''+\lambda_T=\dim T_k=\dim T= t$ and $a_0+t_0+\lambda_A=\dim A_k=\dim A=a+t$.

     Let $\ell$ be a prime big enough and consider the exact sequence of $\ell$-adic Tate modules
    \[
    \begin{tikzcd}
    0 \arrow[r] & (\Tl T)^{I_K} \arrow[r] & (\Tl A)^{I_K}\arrow[r] & (\Tl B)^{I_K} \arrow[r] & 0.
    \end{tikzcd}
    \]

    This exact sequence identifies $\operatorname{T}_{\ell} A_k$ as an extension of $\Tl B_k$ by $\Tl T_k$.  We get the equality
    \[
    t_0+2a_0=t'+t''+2a'.
    \]

    \smallskip

    The variety $A_k$ is an extension of an abelian variety $B_0$ by a commutative group scheme $T_0\times U_0$ where $U_0$ is unipotent of dimension $\lambda_A$. We thus have an exact sequence
    \[
     \begin{tikzcd}
    0 \arrow[r] & \Tl T_0 \arrow[r] & \Tl A_k \arrow[r] & \Tl B_0 \arrow[r] & 0
    \end{tikzcd}
    \]
    where $t_0$ is the dimension of $T_0$ and $a_0$ that of $B_0$.

    \smallskip

    On the other hand, the sequence of Néron models 
    \[
    \mT\to \mA \to \mB
    \]
    is a complex so that the map $\operatorname{T}_{\ell} T_k\to \operatorname{T}_{\ell} B_k$ is zero. Note furthermore that the maps of Tate modules that realize $\Tl A_k$ as an extension of $\Tl B_k$ by $\Tl T_k$ are induced by the maps between Néron models.  We thus get another exact sequence
    \[
     \begin{tikzcd}
    0 \arrow[r] & \Tl T_0/\Tl T_k \arrow[r] & \Tl B_k\arrow[r] & \Tl B_0 \arrow[r] & 0
    \end{tikzcd}
    \]
    from which we get the inequalities
    \[
    t'\geq t_0-t'' \text{ and } a'\geq a_0
    \]

    \medskip

    We thus have the following equalities and inequalities
    \begin{align}
        t'&+a'+\lambda_B=a \label{eq:rangredB}\\
        t''&+\lambda_T=t  \label{eq:rangredT}\\
        t_0&+a_0+\lambda_A=a+t \label{eq:rangredA}\\
        t_0&+2a_0=t'+t''+2a' \label{eq:exactTlA}\\
        t'&\geq t_0-t'' \label{ineq:T_0subB} \\
        a'&\geq a_0 \label{ineq:B0subB}
    \end{align}

    By using equation (\ref{eq:rangredA}) and replacing the values of $t$ and $a$ by those in equations (\ref{eq:rangredT}) and (\ref{eq:rangredB}) we get
    \[
    \lambda_A= t'+t''+a' +\lambda_T+\lambda_B-t_0-a_0.
    \]
    Now using equation (\ref{eq:exactTlA}) we get
    \[
    \lambda_A= \lambda_T+\lambda_B+ a_0-a'.
    \]
    Hence with inequality (\ref{ineq:T_0subB}) and replacing $t'$ by its expression in (\ref{eq:exactTlA}) we get 
    \[
    t_0+2a_0-t''-2a'\geq t_0-t''
    \]
    which simplifies to
    \[
    a_0\geq a'.
    \]
    With the inequality (\ref{ineq:B0subB}) we thus obtained $a_0=a'$ and $\lambda_A=\lambda_T+\lambda_B$. It also follows that $t_0=t'+t''$.
    \end{proof}

We can now state the different characterizations of semi-stable reduction for semi-abelian varieties.

\begin{theo} \label{theo:characsemistable}
    Let $A$ be a semi-abelian variety over $K$. The following statements are equivalent.
    \begin{itemize}
        \item[$(i)$] The semi-abelian variety $A$ has semi-stable reduction.
        \item[$(ii)$] The torus $T$ has good reduction and the abelian variety $B$ has semi-stable reduction.
        \item[$(iii)$] The action of $I_K$ on $\operatorname{T}_{\ell} A$ is unipotent of order $3$ for any $\ell \neq p$.
        \item[$(iii)'$] The action of $I_K$ on $\operatorname{T}_{\ell} A$ is unipotent for any $\ell \neq p$.

    \end{itemize}

    Furthermore, when one of those holds, the toric rank of $A_k$ is $t+t_0$ and its abelian rank is $a_0$ where $B_k$ is an extension of a torus of dimension $t_0$ by an abelian variety of dimension $a_0$.
\end{theo}

\begin{proof}
    If the action of $I$ on the $\ell$-adic Tate module of $A$ is unipotent then its action on $T$ is trivial and the induced one on $B$ is then unipotent of order $2$ by \cite{sga} Exposé IX Proposition~3.5. We get that $(iii)$ is equivalent to $(iii)'$. In both cases, $T$ has good reduction and $B$ has semi-stable reduction so 
    \[
    (iii')\iff(iii)\iff (ii).
    \]
 
    \medskip

    The last equivalence $(i)\iff (ii)$ follows directly from Proposition~\ref{prop:reduccomparrangs} and so does the last part of the statement.
    
\end{proof}

\subsection{The finite monodromy groups of semi-abelian varieties}

\subsubsection{} Let $A$ be a semi-abelian variety over a number field $K$ and $v\in \Sigma_K$ a non-archimedean place above $p$. We will first define finite monodromy groups for $A$. We denote by $K_v^{\mathrm{un}}$ the maximal unramified extension of the completion $K_v$ of $K$ at $v$. Its absolute Galois group $\Gal( \overline{K_v}/ K_v^{\mathrm{un}})$ is the absolute inertia group $I_v$ of $K$ at $v$. 

\begin{lem}
    Let $\ell\neq p$. There is an open subgroup $I_{v,A}\subset I_v$, independent of $\ell$, such that for every open subgroup $I'\subset I_v$ the restriction of the action of $I_v$ on $\Tl A$ is unipotent if and only if $I'\subset I_{v,A}$. Furthermore, $I_{v,A}$ is normal inside $I_v$ and inside $G_{K_v}$.
\end{lem}

\begin{proof}
    Let $I_{v,B}$ be the corresponding subgroup for $B$, which exists by \cite{sga} Exposé~IX, and $I_{v,T}$ the kernel of the action of $I_v$ on $\Tl T$. Let us set $I_{v,A}=I_{v,B}\cap I_{v,T}$. If $I'\subset I_v$ is an open subgroup such that the action of $I'$ on $\Tl A$ is unipotent, we have that $I'$ acts trivially on $\Tl T$ and unipotently on $\Tl B$. We thus have that $I'\subset I_{v,A}$ and conversely.

    \medskip

    Since $I_{v,A}$ is the intersection of normal subgroups in either $G_{K_v}$ or $I_v$ it is itself normal in both groups.
\end{proof}

The main definition of the section is as follows.

\begin{defin} \label{def:monodfinie}
    The finite monodromy group of $A$ at $v$ denoted by $\Phi_{A,v}$ is defined as the quotient $I_v/I_{v,A}$. 
\end{defin}
It follows from the definition that the group $\Phi_{A,v}$ is the Galois group of the smallest extension $K_{v,A}$ of $K_v^{\mathrm{un}}$ such that $A_{K_{v,A}}$ has semi-stable reduction.

\medskip

For a semi-abelian variety over a $p$-adic field we thus have a similar definition allowing us to consider the finite monodromy group of such a variety. 

\medskip

Let us see the basic properties related to these groups.

\begin{prop}
    The group $\Phi_{A,v}$ has the following properties.
    \begin{itemize}
        \item[(i)] The semi-abelian variety $A$ has semi-stable reduction at $v$ if and only if $\Phi_{A,v}=\{1\}$
        \item[(ii)] There is an isomorphism
        \[
        \Phi_{A,v}\simeq \Gamma_p\rtimes \Z/n\Z
        \]
        where $\Gamma_p$ is a $p$-group and $n$ is an integer prime to $p$.
        \item [(iii)] For $L/K$ a finite extension with a place $w\mid v$ such that $A_L$ has semi-stable reduction at $w$ there is an injective map 
        \[
        \Phi_{A,v} \hookrightarrow \Aut (A_L)_{\overline{k(w)}}.
        \]
    \end{itemize}
\end{prop}
\begin{proof}
    Property $(i)$ follows directly from the definition of $\Phi_{A,v}$ and $(ii)$ is a standard fact about inertia groups of $p$-adic fields. 

    \medskip

    For $(iii)$, let us first base change to the maximal unramified extension $K_v^{\mathrm{un}}$ of $K_v$. Consider $A_{K_{A,v}}$ with its canonical descent datum. Note that its reduction is exactly $(A_L)_{\overline{k(w)}}$. By the Néron mapping property the descent datum extends to the Néron model $\mA$ of $A_{K_{A,v}}$ and by pullback to the special fibers it produces an action of $\Phi_{A,v}$ on $(A_L)_{\overline{k(w)}}$. It is left to show that this action is faithful. This last point comes from the following commutative diagram
    \[
    \begin{tikzcd}
        I_v \arrow[r] & \Phi_{A,v} \arrow[r, hook] \arrow[d] & \Aut (\operatorname{T}_{\ell} A)^{I_{v,A}} \arrow[d, "\sim" {rotate=90, anchor=north}] \\
        & \Aut (A_L)_{\overline{k(w)}} \arrow[r, hook] & \operatorname{T}_{\ell} (A_L)_{\overline{k(w)}} 
    \end{tikzcd}
    \]
\end{proof}

Let us finish by recording a characterization which is due to Silverberg and Zarhin in the case of abelian varieties - see \cite{SZ95}~Theorem~5.2. The proof is essentially identical and is omitted.    

\begin{prop}
    Let $\rho_{A,\ell}\colon G_K\to \GL_{2a+t}(\Q_{\ell})$ be the $\ell$-adic representation from a choice of basis on $\operatorname{T}_{\ell}A$. Let $G_v$ be the Zariski closure of $\rho_{A,\ell}(I_v)$ in $\GL_{2a+t} (\Q_{\ell})$. We have an isomorphism
    \[
    \Phi_{A,v} \simeq G_v/G_v^{\circ}.
    \]
\end{prop}

Since $G_v$ has Zariski dense $\Q_{\ell}$-rational points the finite étale group of components $G_v/G_v^{\circ}$ is identified with its group of rational points.

\subsubsection{}  We turn to characterizations of finite monodromy groups of semi-abelian varieties. The first one is geometric and uses semi-abelian varieties over finite fields. We give a group theoretic characterization afterwards.

\begin{prop}\label{prop:caracgeom}
    Let $G$ be a finite ramification group at $p$. Let $A_0$ be a semi-abelian variety of dimension $g$ over a finite field $k$ of characteristic $p$ with a polarization $\lambda_0$ such that there is an injective homomorphism $\iota \colon  G\hookrightarrow \Aut (A_0, \lambda_0)$. Assume furthermore that there is a torus $T_1\subset T_0$ of dimension $t_1$ which is stable by the action of $G$. Then there is a semi-abelian variety $A$ over a $p$-adic field $K$ of dimension $g$ and toric rank $\dim T_1$ such that $G$ is the finite monodromy group of $A$. Furthermore, $A$ is such that for any extension $L$ of $K$ for which $A_L$ has semi-stable reduction, the group $G$ acts faithfully on the reduction $A'_0$ of $A_L$ and stabilizes a subtorus of dimension $t_1$. 
\end{prop}

\begin{proof}
    The injective homomorphism $G\hookrightarrow \Aut (A_0,\lambda_0)$ gives an injective map $G\hookrightarrow \Aut (T_0, \lambda_{T_0})\times \Aut (B_0, \lambda_{B_0})$. The condition that $G$ stabilizes $T_1$ further gives that this injection can be refined in an embedding
    \[
    \widetilde{\iota}\colon G\longrightarrow \Aut T_1 \times \Aut(T_0/T_1\times B_0, \widetilde{\lambda_0})
    \]
    where $\widetilde{\lambda_0}$ is the induced polarization. Let $G_2$ be the image of the projection of $G$ onto the second factor and $G_1$ be the kernel of this map. Since $\widetilde{\iota}$ is an embedding $G_1$ acts faithfully on $T_1$. Let us now fix an extension $L$ of $\Q_p^{\mathrm{un}}$ which is Galois of group $G$ and denote by $L_B$ the subfield fixed by $G_1$ -- this is possible by the main result of \cite{Mau68}.  

    \medskip
    
    Now, by Théorème~3.2 and Théorème~3.10 of \cite{Ph4} the semi-abelian variety $T_0/T_1\times B_0$ lifts to an abelian variety $B$ over the field $K^{\mathrm{un}}$ such that $G_2$ is the finite monodromy group of $B$ and $B_{L_B}$ has semi-stable reduction. We can now consider the semi-abelian variety $A=T\times B$ where $T$ is the torus given by the $L$-twist of $\G_m^{t_1}$ by the map $\operatorname{pr}_1\circ \widetilde{\iota}\colon G\to \operatorname{GL}_{t_1}(\Z)$. By Proposition~\ref{prop:twistorus} the group $G_1$ is a subgroup of the finite monodromy group $\Phi_T$ of $T$ and it follows that $K_A=L$. In particular, we have $\Phi_A=G$. 

    \medskip

    Through its Galois action on $L$ the group $G$ acts on $A_L$ and on its Néron model. This action respects the exact sequence given by Lemma~\ref{lem:neronmodels} and thus, by passing to the special fiber, stabilizes a subtorus of dimension $t_1$. 

    \medskip

    The semi-abelian variety $A$ descends to a $p$-adic field $K$. For a field $F/K$ such that $A_F$ has semi-stable reduction, one can check that the action of $G$ we defined on the base change $A'_{\overline{k}}$ of the reduction $A'$ of $A_F$ as a semi-abelian variety over the residue field $k'$ of $F$ descends to $A'$. 
\end{proof}

This directly provides a geometric characterization of finite monodromy groups of semi-abelian varieties.

\begin{cor}
    Let $G$ be a finite ramification group at $p$. Then $G$ is the finite monodromy group of a semi-abelian variety $A$ of toric rank $t$ and abelian rank $a$ over a number field $K$ at a place $v\in \Sigma_K$ above $p$ if and only if there is a polarized semi-abelian variety $(A_0,\lambda_0)$ of dimension $t+a$ over a finite field of characteristic $p$ with a faithful action of $G$ that stabilizes a torus $T_0\subset A_0$ of dimension $t$.
\end{cor}

\subsubsection{} For a prime $p$ and positive integers $t$, $t'$ and $a$ we introduce the notion of $(p,t,t',a)$-inertial groups. It is a slight generalization of the notion of $(p,t,a)$-inertial groups of Silverberg and Zarhin from \cite{SZ2004}.

\begin{defin}
    Let $G$ be a finite group, $p$ a prime number and $t$, $t_0$ and $a_0$ non-negative integers not all zero. The group $G$ is said to be $(p,t,t',a)$-inertial if the following two conditions hold.
    \begin{itemize}
        \item[$(i)$] The group $G$ is isomorphic to a semi-direct product $\Gamma_p\rtimes \Z/n\Z$ with $\Gamma_p$ a $p$-group and $n$ a positive integer prime to $p$.
        \item[$(ii)$] There exists a family of injective homomorphisms indexed by the prime numbers $\ell\neq p$
        \[
        (\iota_{\ell}\colon G\hookrightarrow \GL_t(\Z)\times \GL_{t'} (\Z) \times \Sp_{2a}(\Q_{\ell}))_{\ell\neq p}
        \]
        such that the first two projections are independent of $\ell$ and the characteristic polynomials of the images of the elements of $G$ by the projections $(p_{\ell}\colon G\to \Sp_{2a'} (\Q_{\ell}))_{\ell\neq p}$ have integral coefficients independent of $\ell$.
    \end{itemize}
\end{defin}

\begin{theo}\label{theo:(p,t,t_0,a_0)-inertial}
    Let $p$ be a prime number, $t$, $t_0$ and $a_0$ be non negative integers not all zero. A finite group $G$ is $(p,t,t_0,a_0)$-inertial if and only if there is a semi-abelian variety $A$ of dimension $t+t_0+a_0$ with toric rank $t$ and abelian rank $t_0+a_0$ over a $p$-adic field $K$ with $G\simeq \Phi_A$ and such that for any field $L/K$ for which $A_L$ has semi-stable reduction, the reduction $A'$ of $A_L$ has toric rank $t+t_0$ and abelian rank $a_0$ and the action of $G$ on $A'$ stabilizes a subtorus of rank $t$. 
\end{theo}
\begin{proof}
    Let $G$ be a $(p,t,t_0,a_0)$-inertial group. Then $G'$, its image on the factor $\GL_{t_0}(\Z)\times \Sp_{2a_0}(\Q_{\ell})$ for some $\ell\neq p$ is independent of $\ell$. The group $G'$ is $(p,t_0,a_0)$-inertial and thus it follows from \cite{Ph7} Theorem~4.8 that there is a semi-abelian variety $A'_0$ with a polarization $\lambda'_0$ such that $G'\subset \Aut (A'_0, \lambda'_0)$. Consider now the semi-abelian variety $A_0=\G_m^t\times A'_0$ with the polarization $\lambda_0=(\mathrm{id\times \lambda'_0)}$. We have an inclusion $G\subset \Aut (A_0,\lambda_0)$ by construction and the induced action stabilizes the subtorus given by the inclusion of $\G_m^t$. By Proposition~\ref{prop:caracgeom} we get the conclusion that $G$ is a finite monodromy group of a semi-abelian variety $A$ with the required properties.

    \medskip

    For the converse, consider the reduction $A'$ of the base change $A_{K_{A,s}}$ of $A$ to $K_{A,s}$, the smallest extension of $K^{\mathrm{un}}$ over which it has semi-stable reduction. The action of $G$ on $A'$ is faithful and stabilizes a subtorus of rank $t$. Let $A_0$ be the reduction of the abelian quotient of $A_{K_{A,s}}$. By the exactness of connected components of Néron models, it is a quotient of $A'$. We have that the induced action of $G$ commutes to any polarization defined over $K$ on the abelian quotient of $A_{K_{A,s}}$ so that we get an injective map
    \[
    G\longrightarrow \GL_t(\Z) \times \Aut (A_0,\lambda_0).
    \]
    As usual, we can decompose further 
    \[
    G\longrightarrow \GL_t(\Z) \times \GL_{t_0}(\Z) \times \Aut (B_0, \lambda_{B_0})
    \]
    by breaking $A_0$ in its toric and abelian exact sequence. Finally, taking $\ell\neq p$ a prime and composing with the $\ell$-adic Tate module map $\Aut(B_0,\lambda_{B_0})\to \Aut T_{\ell} B_0$ we get that $G$ is $(p,t,t_0,a_0)$-inertial.
\end{proof}

As before, the characterization follows directly.

\begin{cor}
    A finite group $G$ is the finite monodromy group of a semi-abelian variety of dimension $g$ over a $p$-adic field $K$ if and only if it is $(p,t,t_0,a_0)$-inertial for some integers $t$, $t_0$, $a_0$ such that $t+t_0+a_0=g$.
\end{cor}

\subsection{The semi-stability degree for semi-abelian varieties} \label{sub:semi-stabdeg}

\subsubsection{} Let us first state the semi-stable reduction theorem for semi-abelian varieties.

\begin{theo} \label{theo:reducsemistable}
    Let $A$ be a semi-abelian variety over a number field $K$. Then, there is a finite extension $L/K$ such that $A_L$ has semi-stable reduction. 

    \medskip

    Furthermore, we have
    \[
    d(A)=\min \{[L:K] \mid A_L \text{ has semi-stable reduction}\}= \underset{v\in \Sigma_K}{\operatorname{lcm}} \Card \Phi_{A,v}.
    \]
\end{theo}

\begin{proof}
    Let $S\subset \Sigma_K$ be the finite set of places of bad reduction of $A$. For $v\in S$ the extension $K_{v,A}$ of $K_v^{\mathrm{un}}$ from Definition~\ref{def:monodfinie} is finite of degree $\Card \Phi_{A,v}$. It descends to a totally ramified extension, that we denote again by $K_{v,A}$, of the same degree of $K_v$ by Lemma~2.1 of \cite{Phi22a}. Let $K'_{v,A}$ be the extension of $K_{v,A}$ obtained by compositum with the unique unramified extension of $K_v$ such that 
    \[[K'_{v,A}:K_v]=\underset{v\in \Sigma_K}{\operatorname{lcm}} \Card \Phi_{A,v}=d.
    \] By Proposition~2.3 of \textit{loc. cit.} there is a finite extension $L/K$ of degree $d$ such that for all places $v\in S$ there is a unique place $w\in \Sigma_L$ above $v$ with a $K_v$-isomorphism of fields $L_w\simeq K'_{v,A}$. By definition $A_L$ has good reduction outside of the set $S'\subset \Sigma_L$ of places above $S$ and, for $w\in S',$ we have by construction
    \[ \Phi_{A_L,w}=\{1\}\]
    so that $A_L$ has semi-stable reduction. We thus have $d(A)\leq d$.

    \medskip

    Let $F/K$ be a finite extension such that $A_F$ has semi-stable reduction. For a place $w\in\Sigma_F$ above $v\in S$ we must then have $F_w^{\mathrm{un}}\supset K_{v,A}$ so that $F_w^{\mathrm{un}}/K_v^{\mathrm{un}}$ is an extension of degree divisible by $\Card \Phi_{A,v}$. It follows that $\Card \Phi_{A,v}$ divides $[F:K]$ and thus $d(A)\geq d$.
\end{proof}

We can now define the degree of stability for semi-abelian varieties with given toric and abelian rank.

\begin{defin} \label{def:degdgta}
    Let $A$ be a semi-abelian variety over a number or local field $K$. We call
    \[d(A)= \min \{ [L:K] \mid A_L \text{ is semi-stable} \}
    \]
    the degree of stability of $A$. 
    
    Further, let $g$ be a positive integer, $t$ and $a$ be non-negative integers with $g=t+a$. We define
    \[
    d_g(t,a)=\max \{ d(A) \mid A \text{ semi-abelian with toric rank }t \text{ and abelian rank } a\}.
    \]
\end{defin}

Let us quickly note that $d_g(t,a)$ is well-defined and give a bound on its value which we will show to be optimal. 

\begin{lem} \label{lem:naivbounddgta}
    Let $A$ be a semi-abelian variety over a number field $K$ with toric rank $t$ and abelian rank $a$. Then we have
    \[
    d(A)\leq d_t(t,0)\cdot d_a(0,a).
    \]
    In particular for $g=t+a$ we have
    \[
    d_g(t,a)\leq d_t(t,0)\cdot d_a(0,a).
    \]
\end{lem}
\begin{proof}
    Let $T$ be the maximal subtorus of $A$ and $B$ the abelian quotient $A/T$. Let $L_1$ and $L_2$ be the extensions of $K$ of degree $d(T)$ and $d(B)$ respectively such that $T_{L_1}$ and $B_{L_2}$ have semi-stable reduction. By Theorem~\ref{theo:characsemistable}, the semi-abelian variety $A_{L_1L_2}$ has semi-stable reduction and thus
    \[
    d(A)\leq d(T)\cdot d(B) \leq  d_t(t,0)\cdot d_a(0,a)
    \]
    by definition.
\end{proof}

In order to compute the value of $d_g(t,a)$ we first have to deal with the analogous computation for tori, in essence that is to compute the value $d_g(g,0)$. We will do this by using the Galois twisting construction.

\subsubsection{} By definition, a torus $T$ of dimension $t$ over a field $K$ is a $\overline{K}$-twist of $\G_m^t$. It is equivalently given by a continuous map $\rho_T\colon G_K\to \Aut \G_m^t=\GL_t(\Z)$ which corresponds to the action of $G_K$ on the character lattice $X(T)$ of $T$. Following Theorem~4.3 of \cite{SZ2004} and its generalization to number fields in \cite{Ph22b}, we can compute the effect of twisting on the places of good reduction of a torus. We deal with the general case of a semi-abelian variety. 

\begin{prop} \label{prop:twistorus}
    Let $A$ be a semi-abelian variety over a number field $K$ and $v\in \Sigma_K$ a finite place. Let $L/K$ be a finite Galois extension and 
    \[
    \iota\colon \Gal (L/K) \longrightarrow \Aut_K A
    \]
    be an injective homomorphism. If $A$ has semi-stable reduction at $v$ then the twist $A'$ of $A$ corresponding to $\iota$ is such that
    \[
    \Phi_{A',v}=I_v(L/K).
    \]
    In particular, for a torus $T$ given by a map
    \[
    \rho_T\colon G_K \longrightarrow \GL_t(\Z),
    \]
    we have $\Phi_{T,v}=I_v(\overline{K}^{\Ker \rho_T}/K)\subset G_K/\Ker \rho_T$.
\end{prop}

\begin{proof}
    Let $\varphi\colon A_L\to A'_L$ be the isomorphism such that $\sigma(\varphi)^{-1}\circ \varphi=\iota(\sigma)$ for every $\sigma\in \Gal(L/K)$. For clarity we denote by $\iota$ again the map $G_K\to \Gal(L/K) \to \Aut_K A$. For $\sigma\in I_v$ and $\ell\neq \operatorname{char}k(v)$ we have  
    \begin{align*}    
    \rho_{A',\ell}(\sigma)&= \sigma(\varphi)^{-1}\circ \rho_{A,\ell} (\sigma) \circ \varphi \\
    &= \iota(\sigma)\circ \varphi^{-1} \rho_{A,\ell} (\sigma) \circ \varphi.
    \end{align*}
    Now, as we have
    \[
    I_{v,A'}=\{ \sigma\in I_v \mid \rho_{A',\ell}(\sigma) \text{ is unipotent} \}
    \]
    we see that $\rho_{A',\ell}(\sigma)$ is unipotent if and only if $\iota(\sigma)$ is trivial since it has finite order.

    \medskip
    
    It follows that
    \[
    \Phi_{A',v}= I_v/I_{v,A'}= I_v/\Ker \iota_{| I_v}= I_v(L/K).
    \]    
\end{proof}

In order to use the previous result to build tori with maximal degree of stability we need to consider the maximal subgroups of $\GL_t(\Z)$ for $t\geq 1$. The list of these groups was given by Feit in \cite{Fei96} whose proof relied on unfinished work of Weisfeiler. In \cite{Ré20} Rémond gives a complete proof of the order of the maximal finite subgroups of $\GL_t(K)$ for cyclotomic fields and $\Q$ in particular.

\medskip

Recall that $\beta\colon \N\to \Q$ is defined by $\beta(2)=3/2$, $\beta(4)=3$, $\beta(6)=9/4$, $\beta(7)=9/2$, $\beta(8)=135/2$, $\beta(9)=15/2$, $\beta(10)=9/4$ and $\beta(n)=1$ for all other values of $n\in \N$. 

\begin{lem} \label{lem:boundnaivtor}
    Let $T$ be a torus of dimension $t$ over a number field $K$. Then we have
    \[
    d(T)\leq \beta(t)2^t t!.
    \]
\end{lem}
\begin{proof}
    By Theorem~\ref{theo:reducsemistable} and by Proposition~\ref{prop:twistorus} we have
    \[
    d(T)=\underset{v\in \Sigma_K}{\operatorname{lcm}} \Card \Phi_{T,v} \mid \Card (G_K/\Ker \rho_T).
    \]
    But since $G_K/\Ker \rho_T$ is a finite subgroup of $\GL_t(\Z)$ the bound follows by Théorème~7.1 of \cite{Ré20}. 
\end{proof}

In order to use Proposition~\ref{prop:twistorus} to build a torus $T$ of dimension $t$ with $d(T)=\beta(t)2^t t!$ we need to solve the Grunwald problem for the list of groups given by Feit. The list is given in Table~\ref{tab:groups_t}. Note that it is claimed that for each dimension the isomorphism class of maximal order finite subgroups is unique. Let $t\geq 1$ and let us denote by $G_t$ a maximal order subgroup of $\GL_t(\Z)$. Then we have $G_t\simeq \Z/2\Z \wr \mfS_t$ for all values of $t$ such that $\beta(t)=1$. For the exceptional cases the groups $G_t$ are described by means of Weyl groups acting on Dynkin diagrams -- see \cite{Bou03} for more on this subject. The only fact we will use is that these groups are generated by pseudo-reflections as matrix groups.
\begin{table}[h]
    \centering
    \renewcommand{\arraystretch}{1.2} 
    \begin{tabular}{cl}
        \toprule
        \textbf{Parameter} $t$ & \textbf{Group} $G_t$ \\
        \midrule
        2 & $G_2 \simeq W(G_2) \simeq D_6$ \\
        4 & $G_4 \simeq W(F_4)$ \\
        6 & $G_6 \simeq W(E_6) \times \Z/2\Z$ \\
        7 & $G_7 \simeq W(E_7)$ \\
        8 & $G_8 \simeq W(E_8)$ \\
        9 & $G_9 \simeq W(E_8) \times \Z/2\Z$ \\
        10 & $G_{10} \simeq W(E_8) \times W(G_2) \simeq W(E_8) \times D_6$ \\
        \bottomrule
    \end{tabular}
    \caption{List of the isomorphism classes of the groups $G_t$ for corresponding values of $t$.}
    \label{tab:groups_t}
\end{table}

\begin{theo}
    Let $t\geq 1$ be an integer. Let $G\subset \GL_t(\Z)$ be a maximal order subgroup. Then $G$ admits a generic Galois extension. 
\end{theo}
\begin{proof}
    For all values of $t$ such that $\beta(t)=1$ we have that such a maximal group is isomorphic to the wreath product $\Z/2\Z\wr \mfS_t$ given by permutation matrices with coefficients in $\{\pm 1\}$. By Proposition~3.8 of \cite{Ph22b} the result follows for these cases. The remaining exceptional cases follow from Theorem~5.1 of \cite{Sal82}, with the added use of Proposition~5.1.6 of \cite{JLY02} for $t=6$, $9$ and $10$, by the Shephard-Todd theorem, for instance given in \cite{Bou03} section V.5. 
    \end{proof}

We can thus state the Grunwald type result we want.

\begin{prop} \label{prop:gruntori}
    Let $G$ be a maximal order subgroup of $\GL_t(\Z)$ and $k$ a number field. Then for a finite set $S\subset \Sigma_k$ of places given with a set $\{M_v/k_v^{\mathrm{un}}\}_{v\in S}$ of finite local extensions of the maximal unramified extensions at those places there is a finite extension $K/k$ and a finite Galois extension $L/K$ of group $G$ such that the following statements hold.
    \begin{itemize}
        \item[$(i)$] For every place $v\in S$ there is a unique place of $K$ above $v$ and it is unramified.
        \item[$(ii)$] For a place $w\in \Sigma_L$ of $L$ above a place $v\in S$ of residue characteristic $p$, the inertia group $I_w(L/K)$ is a $p$-Sylow of $G$ and the local extension $L_w^{\mathrm{un}}/k_v^{\mathrm{un}}$ is linearly disjoint from $M_v$.
    \end{itemize}
\end{prop}

The proof is essentially the same as the proof of Proposition~3.8 of \cite{Ph22b} and is omitted. The added restriction of avoiding a fixed set of local extensions is harmless by the fact that the maximal pro-$p$-extension of $k_v^{\mathrm{un}}$ has a free pro-$p$ Galois group of countable rank -- see for example Lemma~3.5 of \textit{loc. cit.}

We can finally compute the value of $d_g(t,a)$ for all integers $g\geq 1$, $t$, and $a$ such that $g=t+a$.

\begin{theo}\label{theo:dgta}
    For all integers $g\geq 1$, $t\geq 0$, and $a\geq 0$ such that $g=t+a$ we have 
    \[d_g(t,a)= \beta(t)2^tt!\cdot M(2a).
    \]
\end{theo}
\begin{proof}
    Let $g\geq 1$, $t\geq 0$ and $a\geq 0$ be integers such that $g=t+a$. Let $B$ be an abelian variety of dimension $a$ over a number field $k$ such that $d(B)=M(2a)$, which exists by Theorem~1.2 of \cite{Ph4}. If $t=0$, then $B$ suits our needs.

    \medskip
    
    Otherwise, $t\geq 1$ and let $G$ be a maximal order subgroup of $\GL_t(\Z)$. Let $S$ be the set of places of $k$ containing all the places above the prime divisors of $\Card G$ and of $M(2a)$. Let $\{M_v/k_v^{\mathrm{un}}\}_{v\in S}$ be the set of local extensions $\{k_{B,v}/k_v^{\mathrm{un}}\}_{v\in S}$ as defined by Definition~\ref{def:monodfinie}. Let $L/K$ be extensions of $k$ with the properties $(i)$ and $(ii)$ of Proposition~\ref{prop:gruntori} with regard to the set $S$ and the local extensions $\{M_v\}_{v\in S}$ and the group $G$.

    \medskip

    Now the torus $T$ over $K$ defined as the $L/K$-twist of $\G_m^t$ by the injective homomorphism
    \[
    \iota\colon \Gal (L/K) \hookrightarrow \GL_t(\Z)
    \]
    is such that for a place $v$ of $K$ above $p\in S$ the group $\Phi_{T,v}$ is isomorphic to a $p$-Sylow of $G$. It follows that
    \[
    d(T)=\underset{v\in \Sigma_K}{\operatorname{lcm}} \Card \Phi_{T,v} = \Card G= \beta(t)2^tt!.
    \]
    Consider the semi-abelian variety $A=T\times B$ over $K$. The base-change $A_F$ for a finite extension $F/K$ has semi-stable reduction at a place $w\in \Sigma_F$ if and only if $B_F$ and $T_F$ have semi-stable reduction at that place. Let us fix $v\in S$. Since by construction the extensions $k_{B,v}$ and $k_{T,v}$ are linearly disjoint for all places $v\in S$ it follows that $k_{A,v}$, the compositum $k_{B,v}k_{T,v}$, is of degree
    \[
    \Card \Phi_{A,v}=\Card \Phi_{T,v}\cdot \Card \Phi_{B,v}
    \]
    over $k_v^{\mathrm{un}}$. Hence we get
    \[
    d(A)=\underset{v\in \Sigma_K}{\operatorname{lcm}} \Card \Phi_{A,v}=d(T)\cdot d(B)=\beta(t)2^tt!\cdot M(2a).
    \]
    
    \medskip

    In order to conclude for the value of $d_g(t,a)$ remark that by the bound of Lemma~\ref{lem:naivbounddgta} we have a double inequality
    \[
    d(A)\leq d_g(t,a) \leq d_t(t,0)\cdot M(2a). 
    \]
    Since $d_t(t,0)\leq \beta(t)2^tt!$ by Lemma~\ref{lem:boundnaivtor}, both inequalities are equalities. 
\end{proof}

\begin{rem}
    In contrast to the case of abelian varieties if we replace $d_g(t,a)$ by using the lowest common multiple instead of a maximum  
    \[
    d'_g(t,a)=\operatorname{lcm} \{ d(A) \mid A \text{ semi-abelian with toric rank }t \text{ and abelian rank } a\}.
    \]
    we get a different result. To be precise, we have
    \[
    d'_g(t,a)=M(t)\cdot M(2a)
    \]
    where $\beta(t)2^tt!\mid M(t)$. This follows easily by adapting the Galois twisting construction used in this section to construct tori with the $p$-Sylow subgroups of the groups used by Minkowski to compute $M(t)$ as finite monodromy groups.

    \medskip

    The value $d'_g(t,a)$ has also another meaning. Let $\mA(t,a)$ be the set of isomorphism classes $[A]$ of abelian varieties of dimension $g$ over number fields such that if $v$ is a place of bad reduction then the semi-stable toric rank of $A$ at $v$ is at least $t$. Then we have
    \[
    d'_g(t,a)=\max \{d(A) \mid [A] \in \mA(t,a)\}.
    \]
\end{rem}

\subsection{An ad hoc construction at $2$} \label{sub:adhocabvar}

\subsubsection{} We provide a new construction of an abelian variety with maximal finite monodromy at $2$. This variety is shown to be isomorphic to a Jacobian in section~\ref{sub:curvesconstr}.

\medskip

We start by a construction by deformation in the spirit of \cite{Ph4} but only in dimension $1$. Let $E_0/\F_4$ be a supersingular elliptic curve with an embedding $Q_8\subset \Aut E_0$. The associated Fontaine $\varphi$-module $(D,\varphi)$ is simple of rank $2$ with slope $1/2$ and is polarized.

\begin{lem} \label{lem:curveellipticdeform}
    Let $K$ be a $2$-adic field whose residue field contains $\F_4$ and such that there is a totally ramified Galois extension $L/K$ of group $Q_8$. Then there is a finite unramified extension $K'$ of $K$ with an elliptic curve $E/K'$ such that $\Phi_E\simeq Q_8$, $E_{LK'}$ has good reduction and the reduction of $E_{LK'}$ is geometrically isogenous to $E_0$.
\end{lem}

\begin{proof}
    Let $K_0$ be the maximal unramified subextension of $K$. Consider the diagonal action $\delta$ on $D_{L}=D\otimes_{K_0} L$ by $\Q_8$ given by $\rho(h)\otimes h^{-1}$ for $h\in Q_8$. 

    \medskip

    The fixed space $W=(D_L)^{\delta(Q_8)}$ is a $K$-vector space with a natural isomorphism from $W_L$ to $D_L$. Let $F_0\subset W$ be a line and set $F=F_0\otimes_K L\subset D_L$. The line $F$ is Lagrangian in $D_L$ and the filtered module $(D,F)$ is admissible. By base change to the maximal unramified extension of $K$ we are in the situation to apply Proposition~3.5 of \cite{Ph4}. By Serre-Tate's deformation theorem, we get the desired curve $E$ over a finite unramified extension of $K$. Note that the descent datum induces the original faithful action of $Q_8$ on the Fontaine $\varphi$-module of the reduction. We are thus in the situation of Theorem~3.9 of \cite{Ph4} so that $\Phi_E\simeq Q_8$.   
\end{proof}

\subsubsection{} \label{subsub:ellcurves} In order to build an abelian variety we first need to setup the situation to apply Lemma~\ref{lem:curveellipticdeform} many times.

\medskip

Let $n\geq 0$, the $2$-Sylow of $\mathfrak{S}_n$ admits a decomposition as a direct product of iterated wreath product following the binary expansion of $n$ in the following way. Let us write $n=\sum_{i=1}^r 2^{k_i}$ with $k_1> k_2> \dots > k_r\geq  0$. Let $W_k$ be the group defined by induction with $W_0=\{1\}$ and $W_{k+1}=W_k\wr \Z/2\Z $ for $k\geq 0$. Then it is classical, see \cite{Kal48}, that a $2$-Sylow of $\mfS_n$ is isomorphic to the direct product
\[
P_n=W_{k_1} \times W_{k_2} \times \cdots \times W_{k_r}.
\]
The $2$-Sylow subgroups of the wreath product $Q_8\wr \mfS_n$ thus have a similar decomposition as $Q_8\wr P_n$. By the main result of \cite{Mau68}, choose a finite extension $K$ of $\Q_2$, whose residue field contains $\F_4$, such that there is a totally ramified extension $L/K$ with $\Gal(L/K)\simeq G_n$, the $2$-Sylow of $Q_8\wr \mfS_n=Q_8\wr P_n$. Consider the decomposition $\sqcup_{j=1}^r O_j$ of $\{1,\dots, n\}$ in orbits under $P_n$. For each $j\in \{1,\dots, r\}$, choose $\alpha_j\in O_j$ and define
\[
H_j=\operatorname{Stab}_{G_n}(\alpha_j)=Q_8^n\rtimes \operatorname{Stab}_{P_n}(\alpha_j).
\]
The coordinate map $\pi_j\colon (h_1,\dots, h_n, \sigma)\mapsto h_{\alpha_j}$ is surjective with kernel $N_j$. 

\medskip

For $j\in \{1,\dots, r\}$, let $K_j=L^{H_j}$ and $L_j=L^{N_j}$. By definition, $[K_j:K]=\Card O_j$ and $\Gal(L_j/K_j)\simeq H_j/N_j\simeq Q_8$. Lemma~\ref{lem:curveellipticdeform} applied with $L_j/K_j$ provides a curve $E_j$ over an unramified extension $K'_j$ of $K_j$. Choosing a sufficiently large unramified extension $K'$ of $K$ over which we base change all curves and extensions, we can consider, up to renaming, the curves to be over $K_j$, such that their base change to $L_j$ has good reduction and $\Phi_{E_j}\simeq Q_8$. 

\subsubsection{} We can now build our abelian variety by using Weil restriction of scalars. Let us succinctly recall its properties. For a finite separable extension $L/K$ of degree $d$ and an abelian variety $A/L$ of dimension $g$, the Weil restriction $\operatorname{Res}_{L/K}(A)$ is an abelian variety of dimension $dg$, which verifies
\[
{\operatorname{Res}_{L/K}(A)}_{K^s} \simeq \prod_{\tau\colon L\to K^s} A^\tau.
\]
The Weil restriction commutes with duality and sends principal polarizations to principal polarizations. Moreover, for every prime $\ell \neq \operatorname{char}K$, we have an isomorphism of Galois modules
\[
V_{\ell}( \operatorname{Res}_{L/K}(A))\simeq \operatorname{Ind}_{G_L}^{G_K} V_{\ell}(A).
\]

\begin{prop} \label{prop:adhocabvar}
    The abelian variety 
    \[
    A=\prod_{j=1}^r \operatorname{Res}_{K_j/K} (E_j)
    \]
    is principally polarized, of dimension $n$, has good reduction over $L$ and we have
    \[
    \Phi_{A}=G_n.
    \]
    In particular, we have
    \[
    v_2(d(A))=v_2(M(2n)).
    \]
\end{prop}

\begin{proof}
    By construction we have
    \[
    \dim A= \sum_{j=1}^r [K_j:K]= n.
    \]
    For fixed $j$, the principal polarization on $E_j$ induces one on the corresponding factor of $A$ and the product one on $A$ is again principal.  

    \medskip

    To compute the finite monodromy group of $A$, let us fix an odd prime $\ell$ and consider for $j\in \{1,\dots,r \}$ the faithful representation
    \[
    \rho_j\colon Q_8\longrightarrow \operatorname{GL}(V_{\ell}(E_j)).
    \]
    By restriction to the quotient of inertia groups $G_n=\Gal(L/K)$, we have
    \[
    V_{\ell}(A)_{\mid I_K} \simeq \bigoplus_{j=1}^r \operatorname{Ind}_{H_j}^{G_g}(\rho_j\circ \pi_j).
    \]
    It is a direct check that this representation is faithful. Since $A_L$ has good reduction we have, as desired
    \[
    \Phi_{A} \simeq G_n.
    \]
\end{proof}

\section{Finite monodromy groups in degenerating families}

In this section, we show that for a degenerating abelian scheme over a $p$-adic field, that is a group scheme $A/S$ with generic fiber an abelian variety, the finite monodromy groups are upper-semi-continuous for the $p$-adic topology. More precisely, the finite monodromy groups of the fibers at closed points where $A$ degenerates are quotients of the finite monodromy groups of the points in a small enough neighborhood. The main technical result to deduce this is a variant of Krasner's lemma.

\subsection{Krasner's lemma for quasi-finite and flat covers}

\subsubsection{} Let $K$ be a complete field with discrete valuation of rank $1$. We start by reviewing some classical properties of morphisms for the induced map on $K$-points. 

\begin{lem} \label{lem:finiteclosed} Let $f\colon X\to Y$ be a finite map of $K$-schemes. Then the induced map of topological spaces $f(K)\colon X(K)\to Y(K)$ is closed.
\end{lem}

\begin{lem} \label{lem:etopen}
    Let $f\colon X\to Y$ be an étale map of $K$-schemes. Then the induced map of topological spaces $f(K)\colon X(K)\to Y(K)$ is open.
\end{lem}

See \cite{Po17} Proposition~3.5.73 for a proof.

\medskip

We need one last lemma about the smooth locus of rational points. 

\begin{lem} \label{lem:smoothK}
    Let $Y$ be a reduced $K$-scheme of finite type. Let $V\subset Y$ be a Zariski open and dense subscheme. Then we have that the $K$-rational points in the closure of the smooth locus of $Y$ are in the closure of $V(K)$ in $Y(K)$, that is
    \[
    \overline{Y^{\mathrm{sm}}(K)}\subset \overline{V(K)}.
    \]
\end{lem}
\begin{proof}
    Let us write $Y=\bigcup_{i=1}^r Y_i$ a decomposition of $Y$ in irreducible subschemes. For $y\in Y^{\mathrm{sm}}(K)$ we have $y\in Y^{\mathrm{sm}}(K)\cap Y_i(K)$ for some $i\in \{1,\dots, r\}$. So that by replacing $Y$ by $Y_i\setminus \bigcup_{i\neq j} Y_j$ we can assume $Y$ to be irreducible.

    \medskip
    
    As $V$ is Zariski dense, the intersection $U=Y^{\mathrm{sm}}\cap V$ is a non empty open. We now show the inclusion $Y^{\mathrm{sm}}(K)\subset \overline{U(K)}$. For this, let us assume by contradiction that the analytic open $D=Y^{\mathrm{sm}}(K)\setminus \overline{U(K)}$ is non-empty. Because $Y$ is an irreducible scheme of finite type, any non-empty open set in the analytic topology of $Y(K)$ is Zariski dense in $Y$. Since $D$ is as such, its Zariski closure is $Y$.

    \medskip

    However, by definition $D$ is contained in the Zariski closed set $Y\setminus U$, indeed we have 
\[
D\subset Y^{\mathrm{sm}}(K) \setminus \overline{U(K)}\subset Y^{\mathrm{sm}}(K) \setminus U(K)\subset (Y\setminus U)(K).
\]

This is a contradiction and completes the proof.

\end{proof}

Let us introduce some notation. Let $f\colon X\to Y $ be a quasi-finite morphisms of $K$-schemes of finite type. For a rational point $y\in Y(K)$, let us denote by $K(X_y)$ the Galois closure of the field extension $L/K$ generated by all points in the fiber of $f$ over $y$. Formally, one has
\[
K(X_y)= \langle L/K \mid \exists x\in X, f(x)=y \text{ and }L=\kappa(x)\subset \overline{K}\rangle.
\]
We consider two types of subsets of the rational points of $Y$. First, for $L/K$ a finite Galois extension let us write
\[
Y_L=\{y\in Y(K) \mid K(X_y)=L\}
\]
and, secondly, for $G$ a finite group we set
\[
Y_G= \{y\in Y(K) \mid \Gal (K(X_y)/K) \simeq G\}.
\]

Let us make a simple remark concerning these subsets. We have the decompositions
\[
Y(K) = \bigsqcup_{G} Y_G= \bigsqcup_L Y_L
\]
and
\[
Y_G= \bigsqcup_{\Gal(L/K)\simeq G} Y_L.
\]

\subsubsection{} We can now state our geometric version of Krasner's lemma for schemes. The statement is a generalization of Theorem~3.4 of \cite{Phi22a}.

\begin{theo} \label{theo:krasner}
    Let $f\colon X\to Y$ be a quasi-finite, flat and separated map of $K$-schemes of finite type. Let $F\subset Y$ be the minimal closed subset such that $X\setminus f^{-1}(F)$ is finite étale over $V=Y\setminus F$. Then, the following hold. 
    \begin{itemize}
        \item[$(i)$] Let $L/K$ be a finite Galois extension. For all $y\in \overline{V_L}\subset Y(K)$ we have $K(X_y)\subset L$. 

        \smallskip
        
        If $K$ is local, we further have that for all $y\in \overline{V_G}\subset Y(K)$, there is a surjective morphism $G\rightarrow \Gal(K(X_y)/K)$.
        
        \item[$(ii)$] There is an integer $n\geq 1$ and finite groups $H_1,\dots H_n$ such that $V(K)=\sqcup_{i=1}^n V_{H_i}$. Furthermore, the sets $V_{H_1},\dots, V_{H_n}$ are open and closed in $V(K)$.

        \smallskip

        If $K$ is local, then there is an integer $m\geq 1$ and finite extensions $L_1,\dots, L_m$ of $K$ such that $V(K)=\sqcup_{i=1}^m V_{L_i}$. Furthermore, the sets $V_{L_1}, \dots, V_{L_m}$ are open and closed in $V(K)$.

        \item[$(iii)$] Let $y\in \overline{Y^{\mathrm{sm}}(K)}$. Then, there is a finite group $G$ such that $y\in \overline{V_G}$.

        \smallskip

        If $K$ is local, then there is a finite extension $L/K$ such that $y\in \overline{V_L}$.
    \end{itemize}
\end{theo}
\begin{proof}     
    Let $y\in \overline{V_L}$ and $x\in X_y$. Let $U_x$ be a Zariski open such that $f^{-1}(y)\cap U_x=\{x\}$. By the étale-local form of Zariski Main Theorem - see \cite{EGA4.4} Corollaire~18.12.3 - there is an étale map $Y'\to Y$ with a point $y'\in Y'$ above $y$ such that $k(y')=k(y)$ and the base change $U_x'=U_x\times_Y Y'$ admits an open decomposition into $X_1'$ and $X_2'$ induced by Zariski opens of $U_x$ with the induced map $f'\colon U_x'\to Y'$ verifying that its restriction to $X_1'$ is finite and ${f'}^{-1}(y') \cap X_2'=\varnothing$. We thus have that ${f'}^{-1}(y')=\{x'\}\subset X_1'$. Further, the map $X_1'\to Y'$ is finite and flat by base change and restriction to an open. By choosing a neighborhood $W'$ of $y'\in Y'$ we can further assume that the resulting map $f'\colon X_1''\to W'$ is finite, locally free of positive rank. Let $q\colon W'\to Y$ be the induced map. It is étale therefore $q(K)\colon W'(K)\to Y(K)$ is open. It follows that $y'$ is in the preimage $q(K)^{-1}(\overline{V_L})=\overline{q(K)^{-1}(V_L)}$. Hence, $q(K)^{-1}(V_L)$ is not empty. Let $t'\in q(K)^{-1}(V_L)$. By construction the fiber ${f'}^{-1}(t')$ is non-empty and since $t'$ is $K$-rational it is a subscheme of the fiber of $X$ at $q(t')$, which is split by $L$. We thus have the inclusion
    \[
    q(K)^{-1}(V_L)\subset f'(L)(X_1'')(L)
    \]
    By Lemma~\ref{lem:finiteclosed}, the set $f'(L)(X_1''(L))$ is closed. It thus contains $y'$ and we have an $L$-rational point in the fiber above $x'$, which gives
    \[
    k(x)\subset k(x')\subset L.
    \]
    As this holds for all $x\in X_y$ we thus get $K(X_y)\subset L$ as desired.

    \smallskip

    If $K$ is local, we further have that there are only finitely many Galois extensions $L/K$ with given finite group $G$ so that 
    \[
    \overline{V_G}=\bigcup_{\Gal(L/K)\simeq G} \overline{V_L}.
    \]
    For $y\in \overline{V_G}$, we thus have a Galois extension $L/K$ of group $G$ such that $K(X_y)\subset L$, which in turn gives a surjective map $G\to \Gal(K(X_y)/K)$.

    \medskip

    Part~$(ii)$ follows directly from Theorem~3.4 of \cite{Phi22a}. 

    \medskip

    For the last part of the statement, first note we can assume $f$ to be generically étale by passing to reduced subschemes so that $V\subset Y$ is dense and open. The result now follows from Lemma~\ref{lem:smoothK} applied with $V$ and the finite decomposition of $V(K)$ in open and closed subsets coming from $(ii)$.
\end{proof}

For inertia groups of local fields we get the following corollary. 

\begin{cor} \label{cor:krasnerinertie}
    Assume $K$ is local and let $f\colon X\to Y$ be a quasi-finite, flat and separated morphism of $K$-schemes of finite type. Let $F\subset Y$ be the closed subset such that $X\setminus f^{-1}(F)$ is finite étale over $V=Y\setminus F$. Then, there is an integer $n\geq 1$, finite groups $H_1,\dots, H_n$ and open and closed subsets $V_{H_1}^{\mathrm{un}},\dots V_{H_n}^{\mathrm{un}}$ of $V(K)$ such that $V(K)=\sqcup_{i=1}^n V^{\mathrm{un}}_{H_i}$, and for all $1\leq i \leq n$, 
    \[
    y\in V_{H_i}^{\mathrm{un}} \iff I(K(X_y)/K)\simeq H_i.
    \]
    For $1\leq i\leq n$ and all $y\in \overline{V_{H_i}^{\mathrm{un}}}\subset Y(K)$ we have a surjective map $H_i\to I(K(X_y)/K)$.

    \smallskip 

    Furthermore, for $y\in \overline{Y^{\mathrm{sm}}(K)}$ there is an $i\in \{1,\dots, n\}$ such that $y\in \overline{V_{H_i}^{\mathrm{un}}}$.
\end{cor}
\begin{proof}
    Let us first remark that over the finite unramified extension $K'$ of degree $d!$, where $d$ is the degree of the induced finite étale cover of $V$, of $K$, we have $\Gal(K'K(X_y)/K')=I(K(X_y)/K)$ for all $y\in V(K)$. 
    By $(ii)$ of Theorem~\ref{theo:krasner}, there are finite groups $H_1,\dots H_n$ such that $V(K')$ is covered by the open and closed subsets $V_{H_1},\dots V_{H_n}$. As the inclusion $\iota \colon V(K)\to V(K')$ is continuous, the sets $\iota^{-1}(V_{H_1}),\dots \iota^{-1}(V_{H_n})$ are open and closed in $V(K)$ and we have
    \[
    V(K)=\bigsqcup_{i=1}^n \iota^{-1}(V_{H_i}).
    \]
    By construction, these subsets have the property that, for $1\leq i\leq n$,
      \[
    y\in \iota^{-1}(V_{H_i}) \iff I(K(X_y)/K)\simeq H_i.
    \]

    \medskip

    Let us fix $i\in \{1,\dots, n\}$ and consider $y\in \overline{\iota^{-1}(V_{H_i})}$. First, note that we have a decomposition 
    \[
    \iota^{-1}(V_{H_i})=\bigsqcup_{I(L/K)\simeq H_i} V_L.
    \]
    Since $K'$ is local, this is a finite union and thus we have $y\in \overline{V_L}$ for some finite Galois extension $L/K$ with inertia group isomorphic to $H_i$. It follows from part $(i)$ of Theorem~\ref{theo:krasner} that there is an inclusion $K(X_y)\subset L$ and thus a surjective map $\Gal(L/K)\to \Gal (K(X_y)/K)$ and, in particular, a surjection $H_i\simeq I(L/K)\to I(K(X_y)/K)$.

    \medskip

    For the last part of the statement, let $y\in Y^{\mathrm{sm}}(K)$. As $K$ is local, by $(iii)$ of Theorem~\ref{theo:krasner}, there is a finite extension $L/K$ such that $y\in \overline{V_L}$. For $i\in\{1,\dots, n\}$ such that $V_L\subset \iota^{-1}(V_{H_i})$ we get $y\in \overline{\iota^{-1}(V_{H_i})}$ as desired.  
\end{proof}

\subsection{Degenerating families of abelian varieties}

\subsubsection{} Let us first show how finite monodromy groups in the fibers of semi-abelian schemes can be recovered from the $\ell$-torsion cover of the base for $\ell$ big enough. This generalizes Proposition~2.2 of \cite{Phi22a}. 

\begin{lem}\label{lem:bigltorfinimonod} Let $A\to S$ be a semi-abelian scheme over $K$. Let $\ell > \max (2\dim A+1, p)$ be a prime. Then, for every $s\in S(K)$, we have that either $\Card \Gal( K^{\mathrm{un}}(A_s[\ell])/K^{\mathrm{un}})$ is divisible by $\ell$ in which case
\[
\Phi_{A_s} \simeq \Gal( K^{\mathrm{un}}(A_s[\ell])/K^{\mathrm{un}})/ \Z/\ell\Z  
\]
or $\Phi_{A_s} \simeq \Gal( K^{\mathrm{un}}(A_s[\ell])/K^{\mathrm{un}})$.
\end{lem}
\begin{proof}
    We consider the Galois extension $K^{\mathrm{un}}(A_s[\ell])/K^{\mathrm{un}}$ and denote by $G$ its Galois group. By a result of Raynaud the field $K^{\mathrm{un}}_{A_s}$ is a subfield of $K^{\mathrm{un}}(A_s[\ell])$ which is determined by a subgroup $G_{A_s}$ of $G$ and we have $\Phi_{A_s}\simeq G/G_{A_s}$.

    \medskip

    As a Galois module $A_s[\ell]$ is an extension of $T_s[\ell]$ and $B_s[\ell]$. Since $T_s$ has split good reduction over $K^{\mathrm{un}}_{A_s}$, it follows that $G_{A_s}$ acts trivially on $T_s[\ell]$. On the other hand, $G_{A_s}$ acts unipotently with order at most $2$ on $B_s[\ell]$ so that it acts unipotently with order at most $3$ on $A_s[\ell]$. Hence the exponent of $G_{A_s}$ divides $\ell$, and since it is tamely ramified $G_{A_s}$ is either the trivial group or $\Z/\ell\Z$.

    \medskip

    By the choice of $\ell$, we also have that it is prime to the order of $\Phi_{A_s}$ so that the conclusion follows.
\end{proof}

\subsubsection{} We will now apply our Krasner's lemma to obtain the main result of this section. 

\begin{theo}\label{theo:recouvkrasnersab}
    Let $S$ be an irreducible $K$-scheme of finite type and $G\to S$ be a semi-abelian scheme with generic fiber an abelian variety. Let $U\subset S$ be the largest open subscheme over which $G\to S$ is abelian and set $F=S\setminus U$. Then there is an integer $n\geq 1$, finite groups $H_1,\dots, H_n$ and a finite covering $U_1,\dots, U_n$ of $(S\setminus F)(K)$ by non-empty disjoint open subsets such that, for $1\leq i\leq n$,
    \[
    s\in U_i \iff \Phi_{G_s}\simeq H_i,
    \]
    and, if a $K$-rational point $x\in F$ is in the closure $\overline{U_i}$ of $U_i$ for some $i\in \{1,\dots, n\}$ then we have that $\Phi_{G_x}$ is a quotient of $H_i$.
    
    \medskip

    Furthermore, if $x\in F(K)$ is in the closure of the smooth locus $\overline{S^{\mathrm{sm}}(K)}$ then there is an $i\in \{1,\dots, n\}$ such that $x\in \overline{U_i}$.
\end{theo}
\begin{proof}
    Let $\ell> \max(2\dim  G+1,p)$ be a prime. The open sets given by Corollary~\ref{cor:krasnerinertie} applied with the $\ell$-torsion scheme $G[\ell]\to S$, which is flat as $\ell$ is invertible on $S$, are seen to have the required properties by Lemma~\ref{lem:bigltorfinimonod}.
\end{proof}

\section{An application to the stability degree of curves}

\subsection{Stability degree of stable curves}

\subsubsection{} Stable curves have been introduced by Deligne and Mumford in \cite{DM72}. 

\begin{defin}[\cite{DM72} Definition~1.1]
    Let $S$ be a scheme. A stable curve over $S$ is a proper flat morphism $\pi\colon C\to S$ whose geometric fibers are reduced, connected $1$-dimensional schemes such that they have only ordinary double points and the non-singular rational components meet the other components in at least $3$ points. 
\end{defin}

It comes with the definition of stable reduction of smooth proper curves in order to compactify the moduli space of curves.

\begin{defin}[\cite{DM72} Definition~2.2]
    Let $C$ be a curve over a number field $K$. The curve $C$ has stable reduction if there is a stable curve $\mathcal{C}$ over $\Spec \OO_K$ with generic fiber $C$. 
\end{defin}

The main characterization for our purpose is via the Jacobian variety of the curve.

\begin{theo}[\cite{DM72} Theorem~2.5]
 Let $C$ be a curve over a number field $K$. The Jacobian $J(C)$ has semi-stable reduction if and only if $C$ has stable reduction.   
\end{theo}

\subsubsection{} Let us succinctly define the stability degree of stable curves. First, let us state the structure theorem for their Jacobians, which is Example~8 of \cite{BLR90}.
\begin{prop} \label{prop:jacobianstable}
Let $C$ be a stable curve of genus $g$. Let $\widetilde{C}$ be the normalization of $C$ with irreducible components $C_i$. Then the Jacobian $J(C)$ of $C$ is a semi-abelian variety of dimension $g$. In particular, there is a torus $T$ and an exact sequence
\[
\begin{tikzcd}
    0 \arrow[r] & T \arrow[r] & J(C) \arrow[r] & \prod\limits_{i=1}^r J(C_i) \arrow[r] & 0.
\end{tikzcd}
\]
\end{prop}

The definition of stable reduction for a smooth curve is given by Definition~2.2 of \cite{DM72}. They further show that a curve $C$ has stable reduction if and only if its Jacobian $J(C)$ has semi-stable reduction with Theorem~2.4. We can thus further extend this definition to all stable curves.

\begin{defin} \label{def:stablered}
    A stable curve $C$ over a number field $K$ is said to have stable reduction if its Jacobian $J(C)$ has semi-stable reduction as a semi-abelian variety over $K$. 
\end{defin}

\begin{rem}
    Although Definition~\ref{def:stablered} is artificial as it is, it is expected to be equivalent to the more natural extension of the notion of stable reduction to stable curves, that is, the existence of a stable model as in Deligne-Mumford. As it is not our purpose we avoid this issue here. The same issue holds with the stable degree defined next.  
\end{rem}

From stable reduction we can define the stable degree of a curve.

\begin{defin}
    Let $C$ be a stable curve. The stable degree $d(C)$ of $C$ is the semi-stable degree $d(J(C))$ of its Jacobian. 
\end{defin}

\subsection{Stable curves with maximal automorphism groups}\label{sub:curvesconstr}

\subsubsection{} Let us first provide a construction of stable curves in characteristic $2$ with order of automorphism group of maximal $2$-adic valuation. The existence of such curves gives evidence to the expected equality $d_g^{C}=M(2g)$. We first show that the cardinality of the automorphism groups of such curves of genus $g$ have their $2$-adic valuation bounded by $v_2(M(2g))$ for $g\geq 0$. In fact, a result of Lehr and Matignon in an unpublished manuscript show that this bound holds for all primes $p$ and that they are sharp. We record an independent proof here with added details on the equality cases using the dual graph. We also provide constructions for sharpness of the bound.  

\begin{lem} \label{lem:boundautcurvp}
    Let $C$ be a stable curve of genus $g$ over a finite field of characteristic $p$. Then we have 
    \[
    v_p(\Card (\Aut C)) \leq v_p(M(2g)).
    \]
    When equality holds, if $p=2$ or $p=3$ the dual graph of $C$ is a tree and all leaves have genus $1$ with automorphism group of order divisible by $8$ (respectively $3$), if $p\geq 5$ the first Betti number of the dual graph is less than $(p-1)/2$.

    Furthermore, for odd primes $p$ the analogous statement holds in characteristic $0$ and for $p=2$ the similar statement with $v_2(M(2g)/2^{g-1})$ instead of $v_2(M(2g))$ holds.
\end{lem}
\begin{proof}
    Let $\Gamma_C$ be the dual graph of $C$ and denote by $C_1,\dots, C_r$ the normalization of its irreducible components with $D_1,\dots, D_n$ the preimages of the nodal points. Note that, for a given $i$, we consider $C_i$ with its marked points as a stable curve so that $\Aut C_i$ is trivial if $g_i=0$ and if $g_i=1$ we have
    \[
    v_2(\Aut C_i)\leq 3,~v_3(\Aut C_i)\leq 1 \text{ and } v_p(\Card (\Aut C_i))=0 \text{ for } p\geq 5. 
    \]
    The automorphism group of $C$ sits in an exact sequence
    \[
    \begin{tikzcd}
        1 \arrow[r] & \prod_{i=1}^r \Aut C_i
     \arrow[r] & \Aut C \arrow[r] & \Aut \Gamma_C.
    \end{tikzcd}
    \]
    In particular, we have the bound
    \[
    v_p(\Card (\Aut C)) \leq \sum_{i=1}^r v_p(\Card (\Aut C_i))+v_p(\Card(\Aut \Gamma_C)).
    \]
    Let $g_i$ be the genus of $C_i$. By \cite{Sti73}, Satz 1, it is given that, for $g_i\geq 2$ we have a bound 
    \[
    v_p(\Card (\Aut C_i)) \leq \log_p(\frac{4p}{(p-1)^2}g_i^2).
    \]
    Combined with the previous cases of genus $0$ and $1$ we get $v_2(\Aut C_i)\leq  3g_i$ and $v_p(\Card(\Aut C_i))\leq \lfloor 2g_i/(p-1)\rfloor$. We also have that the first inequality is strict for $g_i\geq 2$. 

    \medskip

    Let us consider the automorphisms of the graph $ \Aut \Gamma_C$. We denote by $\Gamma_C^{\circ}$ the core of $\Gamma_C$ obtained by removing vertices of degree one and the incident edges repeatedly. Let $b$ be the first Betti number of $\Gamma_C^{\circ}$. We make cases on the value of $b$.

    \smallskip

    If $b=0$, $\Gamma_C$ is a tree with positive genus leaves. Then its automorphism group acts faithfully on the leaves and we have $\Card(\Aut \Gamma_C)\leq g!$.

    \smallskip

    If $b=1$, then the core is a cycle and the kernel of the action on the vertices of positive genus has order at most $2$. The number of positive genus vertices is at most $g-1$ in this case, so we get $\Card (\Aut \Gamma_C)\leq 2(g-1)!$.

    \smallskip

    If $b\geq 2$, there are at most $g-b$ positive genus vertices. Restriction to the core gives a map $\Aut \Gamma_C\to \Aut \Gamma_C^{\circ}$ whose kernel acts faithfully on the positive genus vertices. By the main result of \cite{NS24} we get
    \[
    \Card (\Aut \Gamma_C^{\circ}) \leq 12 \text{ if } b=2 \text{ and } \Card (\Aut \Gamma_C^{\circ}) \leq 2^b b! \text{ if } b\geq 3. 
    \]

    \medskip

    We now prove the desired bounds, starting with $p=2$. From the different inequalities depending on $b$ we get
    \[
    v_2\bigl(\Card(\Aut \Gamma_C)\bigr)\leq b+v_2(g!).
    \]
    From which we have
    \[
    \begin{aligned}
        v_2\bigl(\Card(\Aut C)\bigr)
        &\leq 3(g-b)+b+v_2(g!)\\
        &=3g+v_2(g!)-2b\\
        &=v_2(M(2g))-2b.
    \end{aligned}
    \]
    We directly get that equality holds only if $b=0$. By the strict inequality from Stichtenoth bound if $g_i\geq 2$ we further get that equality only holds when all leaves have genus $1$, with automorphism group of order divisible by $8$.

    \medskip

    For $p=3$, we get
    \[
    \sum_{i=1}^r v_3\bigl(\Card (\Aut C_i)\bigr) \leq g-b. 
    \]
    Together with the previous estimates depending on $b$ we get the desired inequality with $v_3(M(2g))=g+v_3(g!)$ and that it is strict if $b\geq 1$. It follows, again, that in the equality case the leaves are genus $1$ curves with automorphism group of order divisible by $3$.

    \medskip

    Finally, for $p\geq 5$, let $h=\lfloor (p-1)/2\rfloor$. Let $V_1$ be the set of vertices of genus at least $h$ and $V_0$ be the set of vertices of genus strictly between $0$ and $h$. For a vertex $v$ denote by $g_v$ its genus. Let
    \[
    X=\sum_{v\in V_1} \big\lfloor \frac{g_v}{h} \big\rfloor \text{ and } m=\Card V_0.
    \]
    We have
    \[
    \sum_{i=1}^r v_p\bigl(\Card (\Aut C_i)\bigr) \leq X
    \]
    by definition of the right hand side. Now, remark that the automorphism group $\Aut \Gamma_C$ preserves both sets $V_0$ and $V_1$ and that the subgroup fixing all positive genus vertices has $p$-valuation at most $v_p(b!)$ using the bounds given by \cite{NS24}. So, we get
    \[
    v_p\bigl(\Card(\Aut(\Gamma_C))\bigr)\leq v_p(X!)+v_p(m!)+v_p(b!).
    \]
    We thus obtain
    \[
    v_p\bigl( \Card (\Aut C)\bigr) \leq X+v_p(X!)+v_p(m!)+v_p(b!).
    \]
    Recall that we must have $hX+m+b\leq g$. Write $g=hq+s$ with $0\leq s < h$. Let $d=q-X$ so that $m+b\leq hd+s$. If $d=0$ then $m+b\leq s<p$ so that $v_p(m!)=v_p(b!)=0$. If $d\geq 1$, Legendre's formula gives
    \[
    v_p(m!)+v_p(b!)\leq \frac{b+m}{p-1} \leq \frac{hd+s
    }{2h} < d.
    \]
    Thus, in all cases $v_p(m!)+v_p(b!)\leq d$ with strict inequality if $d>0$. We finally get
    \[
    v_p\bigl( \Card (\Aut C)\bigr) \leq X+v_p(X!)+d\leq q+v_p(q!)=v_p(M(2g)).
    \]
    Equality holds only if $d=0$ which further gives $X=q$ and $b< h$ as claimed.
    \medskip

    In characteristic $0$, the leaves can have at most automorphism groups of order $4$ for $p=2$ which gives the desired bound while for odd $p$ corresponding curves of genus $(p-1)/2$ and automorphism groups of order $p$ are available.
\end{proof}

Recall that for $n\geq 0$, with binary expansion $n=\sum_{i=1}^r 2^{k_i}$ with $k_1> k_2> \dots > k_r\geq 0$ a $2$-Sylow of $\mfS_n$ is isomorphic to the direct product
\[
P_n=W_{k_1} \times W_{k_2} \times \cdots \times W_{k_r}.
\]
\begin{prop}\label{prop:curvemax2}
    For all $g\geq 1$, there is a stable curve $C$ of genus $g\geq 0$ over a finite field of characteristic $2$ such that
    \[
    Q_8\wr P_g\subset \Aut C.
    \]
    In particular, we have 
    \[
    v_2(\Card (\Aut C))= v_2(M(2g)).
    \]
\end{prop}

\begin{proof}
    We provide an inductive way of constructing $C$ through the binary expansion of $g$.

    \medskip

    For $g=1$, we take $C=C_1$ to be the elliptic curve $E$ with automorphism group $Q_8$ and for $g=2$ we glue two copies of this curve at their neutral point which gives $C_2$. In order for the induction step to be clear, we proceed up to $g=4$. For $g=3$, we modify $C_2$ by replacing the glued point by an irrelevant $\PP^1$. The curve $C_2'$ obtained this way recovers $C_2$ by collapsing this irrelevant $\PP^1$. This is represented in figure~\ref{fig:addingirrP1}.
    
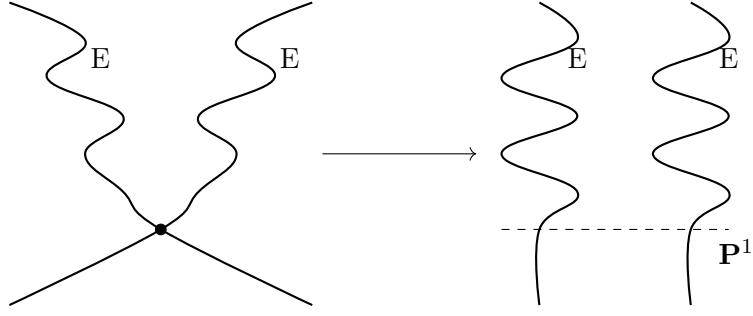
\begin{figure}
    \centering
        \begin{tikzpicture}[scale=1]

    \draw[thick, black] plot [smooth, tension=0.7] coordinates {
        (0,4) (1,3.5) (0.5,3) (1.5,2.5) (1,2) (1.5,1.5) (2,1)  (4,0) 
    };
    
    \node at (1.2,3) [text=black,above, opacity=1] {E};

    \node at (3.7,3) [text=black,above, opacity=1] {E};
    
    \draw[thick, black] plot [smooth, tension=0.7] coordinates {
        (4,4) (3,3.5) (3.5,3) (2.5,2.5) (3,2) (2.5, 1.5) (2,1) (0,0)
    };

    \filldraw[black] (2,1) circle (2pt);

    \node (A) at (4,2) {};
    \node (B) at (6.3, 2) {};
    \draw[->]
  (A) edge (B);

  \draw[thick, black] plot [smooth, tension=0.7] coordinates {
        (7,4) (7.5,3.5) (6.5,3) (7.5,2.5) (6.5,2) (7.5,1.5) (7,1)  (7,0) 
    };
    
    \node at (7.5,3) [text=black,above, opacity=1] {E};

    \node at (9.5,3) [text=black,above, opacity=1] {E};

  \draw[thick, black] plot [smooth, tension=0.7] coordinates {
        (9,4) (9.5,3.5) (8.5,3) (9.5,2.5) (8.5,2) (9.5,1.5) (9,1)  (9,0) 
    };

    \draw[dashed, black] (6.5,1) -- (9.5,1);
    
    \node at (9.6,1) [text=black,below, opacity=1] { $\mathbf{P}^1$};
\end{tikzpicture}

    \caption{Adding an irrelevant $\PP^1$}
    \label{fig:addingirrP1}
\end{figure}
    We then glue $C_1$ to $C_2'$ at any free rational point of the irrelevant $\PP^1$ to obtain $C_3$. 
    
    \medskip
    
    For $C_4$, we proceed in a similar fashion by gluing two copies of $C_2'$ at a given point of the irrelevant $\PP^1$. Here is a graphic representation of $C_4$.
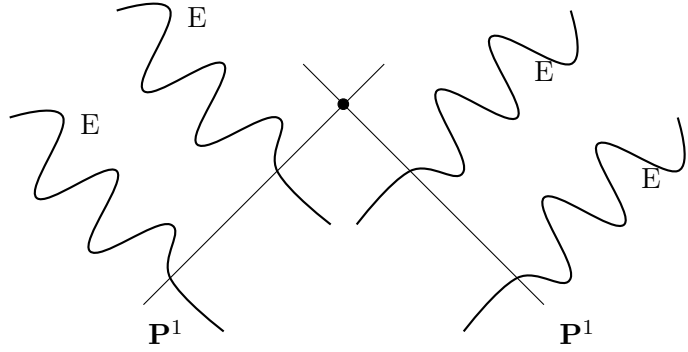
\begin{figure}
    \centering
    \begin{tikzpicture}[scale=1]

    \begin{scope}[rotate around={45:(2,3)}]
        \draw[thick, black] plot [smooth, tension=0.7] coordinates {
            (1,4) (1.5,3.5) (0.5,3) (1.5,2.5) (0.5,2) (1.5,1.5) (1,1)  (1,0) 
        };
        
        \node at (1.5,3) [text=black,above, opacity=1] {E};

        \node at (3.5,3) [text=black,above, opacity=1] {E};
        
        \draw[thick, black] plot [smooth, tension=0.7] coordinates {
            (3,4) (3.5,3.5) (2.5,3) (3.5,2.5) (2.5,2) (3.5,1.5) (3,1)  (3,0) 
        };

        \draw[ black] (0.5,1) -- (5,1);
        
        \node at (0.4,1) [text=black,below right, opacity=1] {$\mathbf{P}^1$};
    \end{scope}

    \begin{scope}[rotate around={-45:(8,3)}]
        \draw[thick, black] plot [smooth, tension=0.7] coordinates {
            (7,4) (7.5,3.5) (6.5,3) (7.5,2.5) (6.5,2) (7.5,1.5) (7,1)  (7,0) 
        };
        
        \node at (7.5,3) [text=black,above, opacity=1] {E};

        \node at (9.5,3) [text=black,above, opacity=1] {E};
        
        \draw[thick, black] plot [smooth, tension=0.7] coordinates {
            (9,4) (9.5,3.5) (8.5,3) (9.5,2.5) (8.5,2) (9.5,1.5) (9,1)  (9,0) 
        };

    \filldraw[black] (5.75,1) circle (2pt);

        \draw[ black] (5,1) -- (9.5,1);
        
        \node at (9.6,1) [text=black,below right, opacity=1] {$\mathbf{P}^1$};
    \end{scope}

\end{tikzpicture}
\caption{The curve $C_4$}
    \label{fig:C4}
\end{figure}

    \medskip

    Before moving to the general induction step we explain how to deal with powers of $2$. Let $g=2^n$. For $n=0$, $1$ and $2$ we have already provided curves $C_1$, $C_2$ and $C_4$. Note that the gluing of $C_4$ was made along the two irrelevant $\PP^1$'s of the copies of $C_2'$. We can thus consider $C_4'$ obtained by replacing the glued point by yet another irrelevant $\PP^1$. For $n\geq 2$, we can thus assume that we have curves $C_{2^{n-1}}$ and $C_{2^{n-1}}'$, where $C_{2^{n-1}}$ is obtained by collapsing an irrelevant $\PP^1$ in $C_{2^{n-1}}'$ and its automorphism group is as desired. We thus glue two copies of $C_{2^{n-1}}'$ at their irrelevant $\PP^1$'s to construct $C_{2^n}$. We readily see that the curve has as automorphism group the wreath product $\Aut C_{2^{n-1}}\wr \Z/2\Z$ and that we can form $C_{2^{n}}'$ by replacing the glued point by an irrelevant $\PP^1$. This process is similar to the one represented in figures~\ref{fig:addingirrP1} and \ref{fig:C4}. Note that the curve we construct has dual graph a complete binary tree with leaves being given by supersingular elliptic curves.

    \medskip

    Let $g\geq 5$ be given with binary expansion $g=\sum_{i=1}^r 2^{k_i}$ with $k_1> k_2 > \cdots > k_r> 0$. By the previous step, we have a curve $C_{2^{k_i}}'$ for each $i\in \{1,\dots,r\}$ if $k_r\geq 1$ and the curve $C_1$ if $k_r=0$. We proceed as before, but now in descending order, we glue $C_{2^{k_1}}'$ with $C_{2^{k_2}}'$, if $k_2\geq 1$, along the irrelevant $\PP^1$'s. We obtain a curve $C_g^1$. For $i\in \{2,\dots,r-1\}$ we repeat this process to obtain a curve $C_g^{r-1}$. For the last step, if $k_r\geq 1$, then we repeat the same process. Otherwise, $k_r=0$ and we have to glue $C_g^{r-1}$ with $C_1$. In order to do this, we follow the blueprint of the $g=3$ case. We add an irrelevant $\PP^1$ to $C_g^{r-1}$ as usual and glue $C_1$ directly to some rational point of this one. The result of this construction is a curve of genus $g$ with automorphism group $Q_8\wr P_g$.
\end{proof}

\begin{rem} \label{rem:Le-Ma2}
    In an unpublished manuscript, cited as [Le-Ma2] in \cite{LM06}, Lehr and Matignon show that the bound
    \[
    v_p(\Card (\Aut C)) \leq  r(2g,p)
    \]
    holds for all primes $p$ and stable curves of genus $g$ over characteristic $p$ fields. They also provide a construction to show that the bound is sharp. In particular, the reduction of the smooth curves whose existence is affirmed by Theorem~\ref{theo:dgc} below have their stable reduction at some places of bad reduction that hit the bound. 
\end{rem}

\subsubsection{} We now construct the specific singular curves that we will use to produce the $p$-part of our lower bound for a prime number $p$. These are made using the $C_r$ curves of Section~3.2 of \cite{CP25} with a construction analogous to that of the previous paragraph. Let us first check that the irreducible components of these curves have suitable automorphism groups and Jacobians.

\begin{lem} \label{lem:crbonnejac}
    Let $p\geq 3$ be a prime number and $r\in \{1,\dots, p-2\}$. Then the curve $C_r$ given by the equation
    \[
    y^r(y-1)=x^p
    \]
    has for $p$-Sylow of its automorphism group $\Z/p\Z$ and for Jacobian $J(C_r)$ a CM abelian variety of dimension $(p-1)/2$ such that $\End J(C_r)\supset \Z[\mu_p]$. 
\end{lem}
\begin{proof}
    The dimension of $J(C_r)$ is given by the genus of $C_r$ which is $(p-1)/2$. We have an injective map
    \[
    \Aut C_r \hookrightarrow \Aut J(C_r)
    \]
    so that $\Z/p\Z\subset \Aut J(C_r)$ which gives an inclusion $\Z[\mu_p]\subset \End J(C_r)$. 
\end{proof}

The construction presented here can also serve as a replacement for the CM abelian varieties used in \cite{Ph22b} for odd primes. For $p=2$ we use another ad hoc construction related to the abelian variety produced in Section~\ref{sub:adhocabvar}.

\begin{prop} \label{prop:Crgenusg}
    Let $g\geq 1$ be a positive integer and $p$ a prime number. Then, there is a stable curve $C$ of genus $g$ over $\Q(\mu_p)$ such that
    \[
    v_p(\Card (\Aut C))=v_p(M(2g)) 
    \]
    if $p\geq 3$
    and
    \[
    v_2(\Card (\Aut C))= v_2(M(2g))+1-g
    \]
    otherwise.
    
    Furthermore, its Jacobian $J(C)$ is a semi-abelian variety with abelian quotient a CM abelian variety $A$ such that $\End A\supset \operatorname{M}_n(\Z[\mu_p])$ for $n=\big \lfloor \frac{2g}{p-1}\big \rfloor$. Moreover, the $p$-Sylow of $\Aut C$ acts faithfully on $J(C)$.
\end{prop}
\begin{proof}
    If $g< (p-1)/2$ then we can take $C$ to be a $\PP^1\setminus\{0,1,\infty\}$ diagram of genus $g$ given by \cite{IN97}.

    \medskip
    
    For $g\geq (p-1)/2$, let $n=\big\lfloor \frac{2g}{p-1}\big \rfloor$ and $0\leq r < p-1$ be the rest of the division of $2g$ by $p-1$ so that
    \[
    g=n\cdot \frac{p-1}{2} +\frac{r}{2}.
    \]
    Let us write $v_p(M(2g))$ as the sum
    \[
    n+v_p(n!).
    \]
    For $k\geq 0$, let $C_{p^k}$ be the curve having dual graph the $p$-ary tree of depth $k$ constructed with
    \begin{itemize}
        \item root $\PP^1_{\Q(\mu_p)}$;
        \item child nodes $\PP^1$'s attached at the points of $\mu_p$;
        \item parent nodes attached at $\infty$;
        \item leaves being $C_r$ curves.
    \end{itemize}
      By construction $C_{p^k}$ has as $p$-Sylow of its automorphism group an iterated wreath product of $\Z/p\Z$ which corresponds to the $p$-Sylow of $\Z/p\Z\wr \mathfrak{S}_{p^k}$. We now decompose $n$ in base $p$
      \[
      n=\sum\limits_{i=0}^s a_ip^i.
      \]
      For all $0\leq i\leq s$, we glue $a_i$ copies of the curve $C_{p^i}$ at their roots on a new copy of $\PP^1_{\Q(\mu_p)}$. We then chain those root $\PP^1$'s together. The resulting curve $C$ has by construction an automorphism group which verifies the inequality
      \[
      v_p(\Card (\Aut C))\geq v_p(M(2g))
      \]
      if $p\geq 3$. As given by Lemma~\ref{lem:boundautcurvp} this is a bound for the $p$-adic valuation of the automorphism group of a stable curve so we have equality. 
      
      \smallskip

      For $p=2$, we get an equality
    \[
    v_2(\Card (\Aut C))= v_2(M(2g))+1-g.
    \]
    
    \smallskip
    
    Finally, for $p\geq 5$, if the rest $r$ is non-zero, we add a component made of a $\PP^1\setminus\{0,1,\infty\}$ diagram of genus $r/2$.
    
    \medskip
    
    Then by Proposition~\ref{prop:jacobianstable} the Jacobian $J(C)$ of $C$ fits in an exact sequence
    \[
\begin{tikzcd}
    0 \arrow[r] & T \arrow[r] & J(C) \arrow[r] & \prod\limits_{i=1}^r J(C_i) \arrow[r] & 0.
\end{tikzcd}
    \]
    Now, we see by Lemma~\ref{lem:crbonnejac} that the abelian varieties $J(C_i)$ are either trivial for those irreducible component of $C$ which are isomorphic to $\PP^1$ or are isomorphic CM abelian varieties with endomorphism ring containing $\Z[\mu_p]$ for the remaining $n$ irreducible components. We thus get the claim.

    \medskip

    Finally, we prove that the action of the chosen Sylow subgroup $P_g$ of $\Aut C$ on $J(C)$ is faithful. First, note that by the classical Torelli result, the map $\Aut C_r\to \Aut J(C_r)$ is injective for $p\geq 5$. The action of the $p$-Sylow of $\Aut C$ on the abelian quotient of $J(C)$ is a combination of the permutation action, which is faithful and these injective maps. For $p=2$, $3$, we consider the automorphism group of the elliptic curve $C_r$, that is, the automorphism fixing the glued point.

\end{proof}

\begin{prop} \label{prop:curvadhoc2}
    Let $g\geq 2$ be an integer. There is a curve $C$ of genus $g$ over a $2$-adic field $K$ such that $\Phi(J(C))=G_g$, the $2$-Sylow of $Q_8\wr \mfS_g$.
\end{prop}

\begin{proof}
    We consider the elliptic curves $E_j$ of paragraph~\ref{subsub:ellcurves} and the dual graph $T$ of the curve provided for $p=2$ of Proposition~\ref{prop:Crgenusg}. The graph $T$ has its set of leaves identified with $I=\{1,\dots, g\}$, trivalent internal vertices and carries an action of the $2$-Sylow $P_g$ of $\mathfrak{S}_g$ by its action on its leaves.  

    \medskip
    
    We first construct a curve $X$ over $L$. Let again $\{O_1,\dots, O_r\}$ be the orbits of $I$ under the action of $P_g$. We have a bijection $G_g/H_j\to O_j$ for every $j\in \{1,\dots, r\}$ given by $\gamma H_j\mapsto \gamma(\alpha_j)$. For each $\beta\in O_j$, choose $\gamma_{\beta}\in G_g$ such that $\gamma_{\beta}(\alpha_j)=\beta$ and set
    \[
    E_{\beta}={(E_j)_L}^{\gamma_{\beta}}.
    \]

    Note that, by construction, the product $\prod_{i\in I} E_i$ constructed that way is isomorphic to
    \[
    \prod_{j=1}^r \prod_{\gamma\in G_g/H_j} {(E_j)_L}^{\gamma_{\beta}} 
    \]
    and carries a natural $G_g$-descent datum with respect to $L/K$. 

    \medskip

    Now, consider the curve $X$ with dual graph $T$ where the leaf at position $i\in I$ is $E_i$ and every internal vertex $v$ is given by a copy of $\PP^1_L$. The components are glued on the origin of the elliptic curves and on the points $\{0,1,\infty\}$ of the genus $0$ curves. The curve $X$ is stable and has genus $g$ by construction. The action of $P_g$ on the tree gives an induced action of $G_g$ on the vertices. For $\sigma\in G_g$ and an internal vertex $v$ with associated rational component $R_v$ we have $v\mapsto \sigma v$ and the induced bijection between incident edges provides a unique semi-linear isomorphism
    \[
    R_v\longrightarrow R_{\sigma v}.
    \]
    These isomorphisms verify the cocycle condition and are compatible with the gluings. We thus have a $G_g$-descent datum on $X$. The descent is effective, let $C/K$ be the descended curve. 

    \medskip

    Finally for the Jacobian, we have 
    \[
    J(X)\simeq \prod_{i\in I} E_i \simeq \prod_{j=1}^r \prod_{\gamma\in G_g/H_j} {(E_j)_L}^{\gamma_{\beta}} \simeq A_L
    \]
    where $A$ is the abelian variety of Proposition~\ref{prop:adhocabvar}. The polarization of $J(X)$ is the product and agrees with the one of $A_L$ under the isomorphism. The isomorphism is $G_g$-equivariant and thus provides the desired isomorphism $J(C)\simeq A$ and the computation of the finite monodromy group of $C$.
\end{proof}

\subsection{Galois twisting and bounding $d_g^{\mathrm{Cur}}$}

\subsubsection{} Let us first relate the twist of a curve $C$ and that of its Jacobian $J(C)$ in the case of a twist by a $p$-group with $p\geq 3$ a prime number.

\begin{prop}\label{prop:twistcurve}
    Let $C$ be a stable curve over a number field $K$ of genus $g\geq 2$ and $v\in \Sigma_K$ a place. Let $L/K$ be a finite Galois extension of group $G$ totally ramified at $v$ with $G$ a $p$-group for $p\geq 3$ a prime number, and 
    \[
    \iota\colon \Gal (L/K) \longrightarrow \Aut_K C
    \]
    be an injective homomorphism. Then, assuming that the induced homomorphism
    \[
    \iota^J\colon \Gal (L/K) \longrightarrow \Aut_KJ(C)
    \]
    is injective, we have that the $L/K$-twist $C'$ of $C$ by $\iota$ has its Jacobian isomorphic to the $L/K$-twist $A$ of $J(C)$ by $\iota^J$. Furthermore, if $C$ has stable reduction at $v$ then we have
    \[
    \Card G\mid d(C').
    \]
\end{prop}
\begin{proof}
    It follows by functoriality of the Jacobian that there is a natural map $t\colon \Aut C\to \Aut J(C)$ for $g\geq 2$. We assume that the induced map $\iota^J$ is injective. 

    \medskip

    By Deligne-Mumford we further have that $C$ and $C'$ have stable reduction at $v$ if and only if their Jacobians have semi-stable reduction at $v$. So that if $C$ has stable reduction at $v$, $J(C)$ has semi-stable reduction at $v$ and the twist $J(C')$ has bad reduction at $v$ with $\Phi_{J(C'),v}\simeq G$ by Proposition~\ref{prop:twistorus}. From the definition we thus get $\Card G\mid d(C')$.
\end{proof}

We now apply this twisting construction to the stable curves given by Proposition~\ref{prop:Crgenusg}.

\begin{prop} \label{prop:exisKCi}
    Let $g\geq 1$ be a natural integer and $p_1, \dots, p_k$ be the odd prime divisors of $M(2g)$. Then there is a number field $K_g$ and stable curves $C_{p_1},\dots, C_{p_k}$ of genus $g$ over $K_g$ such that for all $1\leq i \leq k $ we have 
    \[
    v_{p_i}(M(2g))\leq v_{p_i}(d(C_{p_i})).
    \]
    There is also a place $v$ above $2$ of $K_g$ and curve $C_2$ of genus $g$ over the completion $(K_g)_v$ such that 
    \[
    v_2(M(2g)) \leq v_2(d(C_2)).
    \]
    Furthermore, for any positive integer $M$ we can choose $K_g$ such that each curve has $m$ rational points for some $m\geq M$ and the resulting marked curves have trivial automorphism groups for $M\geq N'_g$. 
\end{prop}
\begin{proof}
    To obtain the first part, we argue as in the proof of Théorème~4.5 of \cite{Ph22b} by applying Proposition~\ref{prop:twistcurve} to the curves of Proposition~\ref{prop:Crgenusg}. By the weak approximation result for fields we can also provide a number field, unramified at all relevant primes, for which the curve of Proposition~\ref{prop:curvadhoc2} is over its completion at $2$. 

    \medskip

    For the last part of the statement apply \cite{MB89} enough times by considering incomplete Skolem data over unramified extensions of $K_g$ at all places above $2$ and $p_1,\dots p_k$. 
\end{proof}

\subsubsection{} To conclude we use a universal scheme for curves of genus $g$ with enough marked points and no automorphism. We will make this statement precise. First let us note the following well-known proposition -- see \cite{Knu83}.

\begin{prop} \label{prop:modulischeme}
    Let $g\geq 1$ be a positive integer and $K$ a number field. There is an integer $N_g$ such that for all $m\geq N_g$ the moduli space $\mM_{g,m}$ over $K$ of smooth proper curves of genus $g$ with $m$ marked points is a smooth scheme.   
\end{prop}

\begin{theo} \label{theo:dgc}
    Let $g\geq 1$ be an integer. We have the equality
    \[
    d_g^{\mathrm{Cur}} = M(2g)
    \]
\end{theo}
\begin{proof}
    Let $K_g$ be the field given by Proposition~\ref{prop:exisKCi} with $M=\max{(N_g,N'_g)}$ the bounds from Proposition~\ref{prop:exisKCi} and \ref{prop:modulischeme}. We then consider the universal stable curve $\overline{\mC}/\overline{\mM}_{g,M_g}$. Let us denote by $\widetilde{\mM}_{g,M_g}$ the largest open subscheme of $\overline{\mM}_{g,M_g}$. It is a fine moduli space for stable curves of genus $g$ with $M_g$ marked points and trivial automorphism group. The pullback of $\overline{\mC}$ provides a universal curve $\widetilde{\mC}/ \widetilde{\mM}_{g,M_g}$. 

    \medskip
    
    Consider the semi-abelian scheme $J(\widetilde{\mC})/ \widetilde{\mM}_{g,M_g}$ given by the Jacobian of the universal curve. Let $p_1,\dots, p_k$ be the odd prime divisors of $M(2g)$. By construction there are stable curves $C_{p_1},\dots ,C_{p_k}$ represented by rational points on $\widetilde{{\mM}}_{g,M_g}$ such that $v_{p_i}(M(2g))\leq v_{p_i}(d(C_{p_i}))$ for $1\leq i\leq k $ and a place $v'_2$ with a curve $C_2\in \widetilde{{\mM}}_{g,M_g}((K_g)_{v'_2})$ such that  $v_2(M(2g))\leq v_2(d(C_2))$. In particular, there are places $v'_2,v'_{p_1},\dots v'_{p_k}\in \Sigma_{K_g}$ such that, for all $i\in \{1,\dots, k\}$,
    \[
    v_{p_i}(M(2g))\leq v_{p_i}( \Card\Phi_{J(C_{p_i}),v_i})
    \]
    and
    \[
    v_2(M(2g)) \leq v_2(\Card \Phi_{J(C_2),v_2}).
    \]

    \medskip

    Let us fix $i\in \{1,\dots, k\}$. By Theorem~\ref{theo:recouvkrasnersab} there is a finite covering $(U_j)_{j\in \{1,\dots r\}}$ of $\mM_{g,M_g}((K_g)_{v_{p_i}})$ and finite groups $H_1,\dots, H_r$ such that 
    \[
    s\in U_j \iff \Phi_{J(\mC)_s,v_i}=H_j.
    \]
    Since $\widetilde{{\mM}}_{g,M_g}((K_g)_{v_i})$ is the closure of $\mM_{g,M_g}((K_g)_{v_i})$, which is smooth, we have that $C_{p_i}\times_{K_g} (K_g)_{v_i}$ is represented by a point $s_i\in \widetilde{\mM}_{g,M_g}((K_g)_{v_i})$ in the closure of $U_j$ for some $j$. It follows by Theorem~\ref{theo:recouvkrasnersab} again that $\Phi_{J(C_{p_i}),v_i}$ is a quotient of $H_j$ and thus 
    \[
    v_{p_i}(\Card H_j)\geq v_{p_i}(M(2g)).
    \]
    Let us denote this open set $U_j$ by $U_{p_i}$.
    Arguing in the same way a similar statement holds for $v_2$. That is, there is a non empty open subset $U_2$ of $\mM_{g,M_g}((K_g)_{v_2})$ such that for all $s\in U_2$ we have
    \[
    v_2(M(2g)) \leq  v_2(\Card \Phi_{J(\mC)_s,v_2}).    
    \]

    \medskip

    The collection of the non empty open subsets $U_2,U_{p_1}, \dots , U_{p_k}$ forms an incomplete Skolem data as given by Moret-Bailly in \cite{MB89}. This datum has a solution, that is there is a finite extension $L/K_g$, unramified at the places $v_2,v_{p_1}, \dots, v_{p_k}$ and such that there are places $w_2\mid v_2$, and $w_{p_i} \mid v_{p_i}$ for $1\leq i \leq k$, and a smooth proper curve $C/L$ of genus $g$ with $M_g$ marked points that verifies
    \[
    v_2(M(2g)) \leq v_2( \Card \Phi_{J(C),w_2})
    \]
    and
    \[
    v_{p_i}(M(2g))\leq v_{p_i}(\Card \Phi_{J(C),w_{p_i}}).
    \]
    It follows that 
    \[
    M(2g) \leq d(C).
    \]

\end{proof}

\printbibliography

\end{document}